\documentclass[11pt]{amsart}
\usepackage{amssymb,latexsym,graphicx, amscd}
\usepackage[T1]{fontenc}
\usepackage{float}
\newtheorem{theorem}{Theorem}[section]
\newtheorem{prop}[theorem]{Proposition}
\newtheorem{lemma}[theorem]{Lemma}
\newtheorem{remark}[theorem]{Remark}
\newtheorem{question}[theorem]{Question}
\newtheorem{definition}[theorem]{Definition}
\newtheorem{cor}[theorem]{Corollary}

\newcommand\R{{\mathbb R}}

\newcommand{\CPb}{\overline{\mathbb{CP}^2}}
\newcommand{\CP}{{\mathbb{CP}}^{2}} 

\begin{document}
\title[Fundamental group]{The fundamental group of symplectic $4$-manifolds with  $b^+=1$}

\author[A. Akhmedov]{Anar Akhmedov}
\address{School of Mathematics, University of Minnesota \newline
\hspace*{.375in} Minneapolis, MN, 55455, USA}
\email{akhmedov@umn.edu}

\author[W. Zhang]{Weiyi Zhang}
\address{Mathematics Institute, University of Warwick \newline
\hspace*{.375in} Coventry, CV4 7AL, UK}
\email{weiyi.zhang@warwick.ac.uk}

\date{February 22, 2012. Revised on August 20, 2015.}


\begin{abstract}
In this article we apply the technique of Luttinger surgery to study the complexity of the fundamental group of symplectic $4$-manifolds with holomorphic Euler number $\chi_h=1$. We discuss the topology of symplectic $4$-manifolds with $b^+=1$ and provide various constructions of symplectic $4$-manifolds with $b^+=1$ and prescribed $c_1^2$.
\end{abstract}

\maketitle

\section{Introduction}

Symplectic $4$-manifolds with $b^+=1$ play a very important role in the theory of smooth $4$-manifolds. The first compact example of homeomorphic but non-diffeomorphic orientable $4$-manifolds (Dolgachev surface $E(1)_{2, 3}$ vs rational surface $\CP\#9\CPb$) have $b^+=1$. This was proved by Donaldson in the ground-breaking article \cite{Do} using $SU(2)$ gauge theory. Nowadays, it is still an active and challenging problem to construct symplectic $4$-manifolds which are homeomorphic but not diffemorphic to small rational surfaces. A concise history of constructing exotic smooth structures on simply-connected $4$-manifolds with $b^+=1$ can be found in the introductions of \cite{AP1, AP}. For recent progress on this problem we refer the reader to (\cite{Ak1, AP1, AP}), where the first exotic minimal symplectic $\CP\#k\CPb$ (for $2 \leq k \leq 5$) were constructed. The small exotic symplectic $4$-manifolds with $b^+=1$ also play very important role in understanding the geography of symplectic $4$-manifolds. Many interesting applications of the symplectic $4$-manifolds with $b^+=1$ to the geography of simply connected non-spin and spin symplectic $4$-manifolds can be found in \cite{Gom, ABBKP, AP, AP2} and the references therein. It is also worth mentioning the work in \cite{IKRT}, where the restriction on the fundamental group of the aspherical symplectic $4$-manifolds has been studied by Ib\'a\~nez-K\k{e}dra-Rudyak-Tralle.    

The wall crossing phenomenon of Seiberg-Witten invariant along with Taubes' SW=Gr does provide us abundant $J$-holomorphic curves on symplectic $4$-manifolds with $b^+=1$. This fact is the cornerstone for many similarities between these symplectic manifolds and algebraic surfaces. Especially, this holds true for many properties involving linear systems or family of curves in a given homology class.  

A K\"ahler surface with $b^+=1$ has the geometric genus $p_g=h^{0,2}=h^{2,0}=0$. By the Kodaira embedding theorem, any such surface has to be an algebraic one. Besides several similarities with algebraic surfaces from the geometric aspects stated above, the topology of symplectic $4$-manifolds with $b^+=1$ is quite different from that of algebraic surfaces with $p_g=0$. For example, it is expected that there are infinitely many minimal symplectic manifolds which are homeomorphic but not diffeomorphic to $\CP\# k\CPb$ with $k\ge 2$. But on the other hand, there are only finitely many such algebraic surfaces since the moduli space is a quasi-projective variety. 
 
One of the main goals of the current article is to demonstrate the huge difference between the topology of symplectic $4$-manifolds with $b^+=1$ and that of algebraic surfaces. In this paper, we are mainly interested in comparing the fundamental groups that occur in these two categories, rather than simply connected $4$-manifolds. 

In a seminal paper, published in the Annals of Mathematics in 1995, Gompf \cite{Gom} showed that any finitely presented group $G$ is the fundamental group of some symplectic $4$-manifold. One drawback of the construction is that the symplectic $4$-manifolds in \cite{Gom} are too large. It is then very natural to ask the ``smallest size'' of such symplectic $4$-manifolds with a given fundamental group $G$. Since the first Betti number $b_1(G)$ is determined by the group $G$ as the rank of $\mathbb R\otimes (G/[G, G])$, a natural quantity to describe the size is $b^+$. Our first result is the following


\begin{theorem} \label{main}

For any finitely presented group $G$ which is dual finite torsion, there is a finitely presented group $G'$ with surjective group homomorphism $G'\twoheadrightarrow G$, such that there is a minimal symplectic $4$-manifolds $X(G)$ with $b^+(X(G))=b_1(G)+1$, $\pi_1(X(G))=G'$ and $c_1^2(X(G))=0$. Moreover, $b_1(G)=b_1(G')$.
\end{theorem}

Here we call a finitely presented group $G = \langle x_{1}, \ \cdots, \ x_{k} \ | \ l_{1}, \ \cdots, \ l_{m} \rangle$ is {\it dual finite torsion} if we write the relators $l_i$, after abelianization, as $x_1^{m_{1i}}\cdots x_k^{m_{ki}}$, then the matrix $(m_{ij})_{1\le i\le k, 1\le j\le m}$ has rank $m$.



For most $G$, the smallest value of $b^+$ that is known to occur in symplectic $4$-manifolds with $\pi_1=G$ is $b_1(G)+1$. See later discussion for details.

By applying our Theorem~\ref{main} to the symplectic $4$-manifolds with $b^+=1$, we deduce the following corollary


\begin{cor} \label{fund_b+=1}
Given any finitely presented group $G$ which is dual finite torsion with $b_1(G)=0$ there is a minimal symplectic $4$-manifold with $b^+=1$, $c_1^2=0$ and $\pi_1=G'$ where $G'$ surjects onto $G$ and $b_1(G)=b_1(G')$.
\end{cor}

The groups $G$ with $b_1(G)=0$ are often called $\mathbb Q$-perfect in the literature. 

Most of the symplectic $4$-manifolds that we construct in this paper are non-K\"ahler. Any minimal complex surface with $c_1^2=0$ is an elliptic surface. When $b^+=\chi_h=1$, the fundamental groups are the $2$-orbifold fundamental groups $E_{p_1, p_2, \cdots, p_k}=\{x_1, x_2, \cdots, x_k|x_1^{p_1}=x_2^{p_2}=\cdots=x_k^{p_k}=x_1x_2\cdots x_k=1\}$ (see \cite{FM} for details). While in our construction, the most groups are not of this type. It is also interesting to notice that  $E_{p_1, p_2, \cdots, p_k}$ does not surject to many groups, e.g. the congruence subgroup $K_{3, 4}$ (see Remark \ref{orbifoldgroup}).



Theorem \ref{main} and Corollary \ref{fund_b+=1} suggest the following: for any finitely presented group $G$, there is a minimal symplectic $4$-manifold $X$ with $b^+(X)=b_1(G)+1$, $\pi_1(X)=G$ and $c_1^2(X)=0$. In this paper, we verify this when $G$ is a free product of cyclic groups (Theorem \ref{freegroup}) and a slightly more general case (Remark \ref{moregenerators}). The full version will be discussed in a sequel.


 Our next focus are symplectic $4$-manifolds of general type. After constructing minimal symplectic $4$-manifolds with arbitrary first homology group in Section \ref{sec:c}, we build symplectic examples with small size and prescribed $c_1^2$ and abelian fundamental groups.
 
 \begin{theorem}\label{main2}
For any integer $2\le h\le 3$, there exist minimal symplectic $4$-manifolds $X_{p, q}$ and $X_p$ with $p, q\ge 1$ such that
\begin{enumerate}
\item $\pi_1(X_{p,q})=\mathbb Z_p\times \mathbb Z_q$, $b^+(X_{p,q})=1$ and $c_1^2(X_{p,q})=h$; 
\item $\pi_1(X_{p})=\mathbb Z\times \mathbb Z_p$, $b^+(X_p)=2$ and $c_1^2(X_p)=h$.
\end{enumerate}

For any $p, q\ge 1$ with $\gcd(p, 2q)=1$, there exists minimal symplectic $4$-manifold $X_{p, q}$ such that  
$$\pi_1(X_{p,q})=\mathbb Z_p\times \mathbb Z_q, \quad b^+(X_{p,q})=1,\quad c_1^2(X_{p,q})=1.$$
\end{theorem}

On the other hand, for algebraic surfaces with $c_1^2=1$ (resp. $c_1^2=2$) and $p_g=0$, the algebraic fundamental group has to be finite, and could only have order at most $5$ (resp. $9$). 

The main surgical technique that we employ in our constructions is Luttinger surgery \cite{Lu}. It has been effectively used recently to construct symplectic $4$-manifolds \cite{AP, ABP, BK, FPS, ABBKP} (see also more recent articles \cite{T1, T2, AO} where various interesting constructions are presented). Luttinger surgery introduces the commutator relation to the fundamental group in our examples, and thus it usually changes the fundamental group of the original $4$-manifold. In fact, by carefully choosing the surgery coefficients and performing the fiber summing with other building blocks, we construct symplectic $4$-manifolds with various fundamental groups.

Our article is organized as follows. We first begin by comparing several topological aspects of symplectic $4$-manifolds and algebraic surfaces. We mainly focus on $4$-manifolds with $b^+=1$, although there are discussions on the size of symplectic manifolds with given fundamental group in Section \ref{invariants} and related results as in Proposition \ref{sympNori}. Next, we apply Luttinger surgeries and symplectic fiber sum along $\mathbb{T}^2$ to prove Theorem \ref{main} and also construct other examples with $c_1^2>0$. In the final section, we construct symplectic $4$-manifolds using both symplectic fiber sum and rational blowdown surgeries. This provides an alternative way to obtain symplectic $4$-manifolds with $c_1^2>0$ and cyclic fundamental groups.

\section{Topology of symplectic $4$-manifolds and algebraic surfaces with $b^+=1$}

In this section, we study the topology of symplectic $4$-manifolds with $b^+=1$ and algebraic surfaces with $p_g=0$. Our main focus will be the fundamental group.

\subsection{Complex and Symplectic Kodaira Dimensions}\label{Kodaira}

The notion of Kodaira dimension has been introduced by K. Kodaira for complex manifolds. It has played a fundamental role  
in the classification theory of complex surfaces. 

If the manifold $X$ admits a complex structure $J$, then the
Kodaira dimension is defined as follows (in the case dim$_\R
X=4$):  The n-th plurigenus $P_n(X,J)$ of a complex manifold is defined
by $P_n(X,J)=h^0({\mathcal K}_J^{\otimes n})$, with ${\mathcal
K}_J$ the canonical bundle of $(X,J)$.  

\begin{definition}The complex Kodaira dimension $\kappa^{h}(X,J)$ of a complex surface $X$ is defined as

\[
\kappa^{h}(X,J)=\left\{\begin{array}{cl} 
-\infty &\hbox{ if $P_n(X,J)=0$ for all $n\ge 1$},\\
0& \hbox{ if $P_n(X,J)\in \{0,1\}$, but $\not\equiv 0$ for all $n\ge 1$},\\
1& \hbox{ if $P_n(X,J)\sim cn$; $c>0$},\\
2& \hbox{ if $P_n(X,J)\sim cn^2$; $c>0$}.\\
\end{array}\right.
\]

\end{definition}

For symplectic $4$-manifolds, there is also a notion of symplectic Kodaira 
dimension (see \cite{li, McDuffS, LeBrun}). 

\begin{definition} \label{sym Kod}
For a minimal symplectic $4$-manifold $(X^4,\omega)$ with symplectic
canonical class $K_{\omega}$,   the Kodaira dimension of
$(X^4,\omega)$ is defined in the following way:

$$
\kappa^s(X^4,\omega)=\begin{cases} \begin{array}{lll}
-\infty & \hbox{ if $K_{\omega}\cdot [\omega]<0$ or} & K_{\omega}\cdot K_{\omega}<0,\\
0& \hbox{ if $K_{\omega}\cdot [\omega]=0$ and} & K_{\omega}\cdot K_{\omega}=0,\\
1& \hbox{ if $K_{\omega}\cdot [\omega]> 0$ and} & K_{\omega}\cdot K_{\omega}=0,\\
2& \hbox{ if $K_{\omega}\cdot [\omega]>0$ and} & K_{\omega}\cdot K_{\omega}>0.\\
\end{array}
\end{cases}
$$

If $(X^4,\omega)$ is not minimal, its Kodaira dimension is defined to be that of any of its minimal models. 
\end{definition}

It is proved in \cite{li} that the symplectic Kodaira dimension is a diffeomorphism invariant. 
Also, it was shown in \cite{DZ} that the symplectic Kodaira dimension coincides with the complex Kodaira dimension 
when both are defined.

It is proved by McDuff that all symplectic $4$-manifolds with $\kappa^s=-\infty$ has to be (diffeomorphic to) a rational or ruled surface. Thus, all of them have $b^+=1$ and could be endowed with K\"ahler structure. 

All the other minimal symplectic $4$-manifolds have $c_1^2=K^2\ge 0$. Using this inequality, we obtain several topological restrictions for such symplectic $4$-manifolds. Since $c_1^2(X)=2e(X)+3\sigma(X)\ge 0$, we have $2(2-2b_1+b^++b^-)+3(b^+-b^-)\ge 0$. Here $e(X)$ and $\sigma(X)$ are Euler number and signature of $X$ respectively. In summary, when $b^+=1$, we have $4b_1+b^-\le 9$. Since $b^+=1$ and $b^+-b_1$ has to be odd, $b_1=0$ or $2$. In other word, we either have
\begin{enumerate}
\item $b_1=0$, $b^-=0, 1, \cdots, 9$; or
\item $b_1=2$, $b^-=0$ or $1$.
\end{enumerate}

\subsection{Algebraic surfaces}

Algebraic surfaces with $p_g=0$ are of Class I in Kodaira's classification list. They achieve all possible Kodaira dimensions. Let us assume that the algebraic surfaces are minimal. 

On an algebraic surface $S$, the geometric genus $p_g:=P_1:=h^0(S, K_S)$, along with the irregularity $q:=h^1(\mathcal O_S)$ determines the holomorphic Euler characteristic $\chi_h(S):=\chi(\mathcal O_S):=1-q+p_g$. 

The case $b_1=2$, $b^-=0$ cannot occur for algebraic surfaces with $p_g=0$, since it violates the Bogolomov-Miyaoka-Yau inequality $c_1^2\le 3c_2$ for surfaces of general type. All the rest have algebraic examplars. 

For the other possibility in case (2), i.e. when $b_1=2$, $b^-=1$, a surface could be an $\mathbb S^2$ bundle over $\mathbb T^2$ which has Kodaira dimension $-\infty$, or a hyperelliptic surface which has Kodaira dimension zero. 

For case (1), if $b^-=9$ and minimal, it has to be an elliptic surface. For example, it could be a Dolgachev surface, which is simply connected and has Kodaira dimension one. It also could be Enriques surface, which has Kodaira dimension zero and $\pi_1=\mathbb Z_2$. More generally, the following fundamental groups are realized by $E(1)_{p_1, p_2, \cdots, p_k}$ ($k$ logarithmic transformations with coefficients $p_1, \cdots, p_k$ performed on elliptic $E(1)$): 
$$\pi_1(E(1)_{p_1, p_2, \cdots, p_k})=\{x_1, x_2, \cdots, x_k|x_1^{p_1}=x_2^{p_2}=\cdots=x_k^{p_k}=x_1x_2\cdots x_k=1\}.$$

We denote these groups by $E_{p_1, p_2, \cdots, p_k}$. It is finite only if $k\le 3$. All the finite groups in this family are subgroups of $SO(3)$:  the cyclic groups $\mathbb Z_m$, dihedral groups $D_m$, $A_4$, $S_4$ and $S_5$, corresponding to $(mp, mq)$ with $(p, q)=1$, $(2,2,m)$, $(2,3,3)$, $(2,3,4)$ and $(2,3,5)$ respectively. It is worth mentioning that $E(1)_{2, 2}$ is an Enriques surface and $E(1)_{p, q}$, with $(p, q)=1$, are Dolgachev surfaces.

For the rest of case (1), i.e. $b^-=0, 1, \cdots, 8$, if the surface is not $\CP$ nor $\mathbb S^2\times \mathbb S^2$, it has to be a surface of general type. It is known that the number of irreducible components for these surfaces is finite. This is a corollary of the following theorem of Gieseker \cite{Gie}.

\begin{theorem}\label{Gie}
For every pair of integers $x$, $y$, the moduli space $\mathcal M_{x, y}$ of canonical surfaces with numerical invariants $K^2=x$ and $\chi_h=y$ is a (possibly empty) quasiprojective scheme and then it has a finite number of components. 
\end{theorem}
 Notice in the cases that we are interested, $p_g=q=0$ and $b^-=0, \cdots, 8$, there are only finitely many choices of $x=1, \cdots, 9$ and $y=1$. Thus, there are only finitely many such surfaces up to diffeomorphism. In particular, the number of possible fundamental groups is finite.

For detailed discussion on these surfaces, especially regarding the possible (algebraic) fundamental groups, see the beautiful survey paper \cite{BCP}. One of the outstanding problems asks for the greatest positive number $a(=9-b^-)$ such that $c_1^2\le a$ implies the fundamental group to be finite, and the smallest positive number $b$ such that $c_1^2\ge b$ implies the fundamental group to be infinite. Currently, we know $b\ge 7$ and $a\le 3$. 

There is a remarkable corollary of Yau's solution of Calabi conjecture \cite{Yau}.

\begin{theorem}\label{Yau}
Let $S$ be an algebraic surface of general type. Then $K_S^2=9\chi_h(S)$ if and only if the universal covering of $S$ is the complex ball.
\end{theorem}

Especially, when $p_g=0$, $c_1^2=9 \Longrightarrow |\pi_1|=\infty$.  Recently, Prasad and Yeung (and Cartwright-Steger) give a complete classification of these so called fake projective planes. It is known that there are exactly 100 up to biholomorphism \cite{PY, CS}. 

\subsection{Symplectic $4$-manifolds}
Let us now take a look at the possible list of symplectic $4$-manifolds with $b^+=1$. We discuss it using the notion of symplectic Kodaira dimension.

For $\kappa^s=-\infty$, as we mentioned, there are nothing more than rational and ruled surfaces.

For $\kappa^s=0$, although there are no classification at this moment, in addition to hyperelliptic surfaces, we know that there are certain (non-complex) $\mathbb T^2$ bundles over $\mathbb T^2$ with $b_1=2$. All of them are symplectic \cite{Gei}. They could have different fundamental groups from hyperelliptic surfaces since $Nil^4$, $Nil^3\times \mathbb E^1$, $Sol^3\times \mathbb E^1$ and $\mathbb E^4$ are all the possible geometric types and on the other hand all hyperelliptic surfaces have geometric type $\mathbb E^4$. 

Let us digress on the definition of a geometry structure (in the sense of Thurston).
A {\it model geometry} is a simply connected smooth manifold $X$ together with a transitive action of a Lie group $G$ on $X$ with compact stabilizers.
A model geometry is called {\it maximal} if $G$ is maximal among groups acting smoothly and transitively on $X$ with compact stabilizers. Sometimes this condition is included in the definition of a model geometry.
A {\it geometric structure} on a manifold $M$ is a diffeomorphism from $M$ to $X/\Gamma$ for some model geometry $X$, where $\Gamma$ is a discrete subgroup of $G$ acting freely on $X$. If a given manifold admits a geometric structure, then it admits one whose model is maximal. In other word, different geometries are distinguished by their fundamental groups.

It is remarkable that T.-J. Li \cite{L_Kod0c} proved that all symplectic manifolds with $\kappa^s=0$, SCY surfaces, have the same rational homology types as $K3$ surfaces, Enrique surfaces or $\mathbb T^2$ bundles over $\mathbb T^2$. Moreover, the paper of S. Friedl and S. Vidussi discussed the restrictions on the fundamental groups of SCY $4$-manifolds \cite{FV}.

For $\kappa^s=1$ or $2$, the topology of symplectic $4$-manifolds is far more complicated than algebraic surfaces. As shown by Gompf, any finitely presented group could be realized as fundamental groups for symplectic $4$-manifolds of $\kappa^s=1$, resp. $\kappa^s=2$. When we restrict on $b^+=1$, it is interesting to know 

\begin{question}\label{questionb+=1} Is any finitely presented group $G$ with $G/[G, G]$ finite abelian group, the fundamental group of a symplectic $4$-manifold with $b^+=1$?
\end{question}

We answer this question at homology level in Corollary \ref{fund_b+=1}. Our manifolds are symplectically minimal and all have $\kappa^s=1$. An infinite family of minimal symplectic $4$-manifolds with $b^+=1$, $b_1=2$, $b^-=1$ and $\kappa^s=1$ were constructed in \cite{Ak}.


For symplectic $4$-manifolds with $\kappa^s=2$, there are possibly two cases: $b_1=2$, $b^-=0$ or $b_1=0$, $b^-=0, 1, \cdots, 8$. The first case is not expected to occur in the symplectic category since there are currently no examples beyond the BMY line, i.e. $c_1^2>3c_2$. Hence, we focus on the second case. In this paper, we construct several examples through Luttinger surgery or rational blowdown with fixed $c_1^2$ ($=9-b^-$). This would demonstrate the difference between symplectic and algebraic categories.

Let us discuss an example to savor the difference between symplectic $4$-manifolds and algebraic surfaces on fundamental groups.

In his paper \cite{Nor} on Zariski's conjecture, Nori shows that for any embedded algebraic curve $C$ with $C^2>0$ in an algebraic surface $S$, the induced group homomorphism $\pi_1(C)\rightarrow \pi_1(S)$ is surjective (Corollary 2.4 B). However, such result does not hold in symplectic category.

Let us first recall the definition of the rank of a finitely presented group. This should not be confused with the rank of a vector space.

\begin{definition}
The rank of a finitely generated group $G$, $rank(G)$, is defined as the smallest cardinality of a set $\Theta$ such that there exists an onto homomorphism $F(\Theta)\rightarrow G$, where $F(\Theta)$ is the free group with free basis $\Theta$. 
\end{definition}

It is clear from the definition that if $H$ is a quotient group of $G$ then $rank(H)\le rank(G)$. In addition, according to classic Grushko theorem, rank behaves additively with respect to taking free product, i.e. for any groups $A$ and $B$ we have $rank(A*B)=rank(A)+rank(B)$. Especially, $rank({\Large{*}}_{i=1}^k \mathbb Z_{p_i})=\sum_{i=1}^k rank(\mathbb Z_{p_i})=k$. Here ${\Large{*}}_{i=1}^k \mathbb Z_{p_i}$ is the free product of cyclic groups $\mathbb Z_{p_i}$.


\begin{prop}\label{sympNori} For any finitely presented group $G$, we have a minimal symplectic $4$-manifold $X$ with $K^2=1$, $\pi_1(X)=G$ and $b^+>1$. Moreover, when $rank(G)>4$, there is an embedded symplectic surface $C$ in $X$ such that the induced group homomorphism $\pi_1(C)\rightarrow \pi_1(X)$ is not surjective.
\end{prop}

\begin{proof}
The first statement follows from Theorem 6.2 (B) in \cite{Gom}. In this theorem, the geography of symplectic $4$-manifolds with given fundamental group is discussed. 

By Taubes' theorem \cite{T}, there is an embedded symplectic surface $C$ in the canonical class $K$. By adjunction formula, it has genus two. Since $rank(\pi_1(\Sigma_2))\le 4$ (it is actually true that $rank(\pi_1(\Sigma_2))= 4$), $rank(G)$ has to be no greater than $4$ if the homomorphism  $\pi_1(C)\rightarrow \pi_1(X)$ is surjective. This contradicts our assumption.
\end{proof}

It would be interesting to know whether this is still true when $X$ is an algebraic surface. However, if Conjecture A in \cite{Nor} is true, then $\pi_1(X)$ would be a quotient group of $\pi_1(\Sigma_g)$ where $g=K^2+1$. Hence the fundamental group of the symplectic surface in the above construction would surject to $\pi_1(X)$.

\subsection{Topological fundamental groups and algebraic fundamental groups}
By Lefschetz's hyperplane theorem and Bertini theorem, the fundamental group of a non-singular projective variety of dimension $\ge 3$ equals the fundamental group of a non-singular hyperplane section. Therefore the set of the fundamental groups of projective varieties is the same as that of algebraic surfaces.

The are many restrictions on the fundamental groups of an algebraic surface. For example, $b_1(G)$ of any such group $G$ should be an even number. However, there is a classical result proved by Serre \cite{Ser}.

\begin{theorem}
Any finite group is the fundamental group of some algebraic surface. 
\end{theorem}

\subsubsection{Algebraic fundamental groups}
When we talk about projective varieties, the algebraic fundamental group is a more suitable notion to work with. 

The topological fundamental group could be understood as the group of deck transformation of the universal covering space. However, universal covering space is not very practical to deal with in the algebraic category. On the other hand, finite \'etale morphisms are the appropriate generalization of the covering spaces in the algebraic category. The notion of \'etale fundamental group is introduced by Grothendieck in \cite{SGA1} as a manageable generalization of topological ones. We refer the readers to \cite{SGA1} for a complete definition. However, we would like to point out the following fact: When the base field $K$ is the complex field $\mathbb C$ (or more generally a separably closed field of characteristic $0$), we know the {\it algebraic fundamental group} $\pi_1^{alg}$, as \'etale fundamental group is typically called in this case, is the profinite completion (i.e. the projective limit of all finite quotients) of the topological fundamental group.

\subsubsection{Algebraic fundamental groups for surfaces with $p_g=0$}

When $S$ is a surface of general type with $p_g=0$, the irregularity $q$ has to be $0$ and $\chi_h(S)=1$.

The algebraic fundamental group behaves as two extremes for surfaces of general type with $p_g=0$. When $c_1^2=9$, Theorem \ref{Yau} asserts that the fundamental groups should be always infinite, since all such manifolds are ball quotients.

On the other hand, when $c_1^2=1$ or $2$, there is the following remarkable result of M. Reid \cite{Reid}.
\begin{theorem}\label{reid}
1) $c_1^2=1 \Longrightarrow \pi_1^{alg}\cong \mathbb Z_m$ for $1\le m\le 5$.

2)  $c_1^2=2 \Longrightarrow |\pi_1^{alg}|\le 9$.
\end{theorem}

The statements are sharp in both cases. For the case $c_1^2=1$, all could be realized. For $c_1^2=2$, out of all finite groups with order no greater than $9$, only the groups $\mathbb Z_4$ and $\mathbb Z_2\times \mathbb Z_3$ are not known how to be realized as $\pi_1^{alg}$. It is conjectured that topological fundamental group should also satisfy both bounds (\cite{BCP}). In addition to the cases mentioned, $\mathbb Z_3$ for both, and $\mathbb Z_6$ for $c_1^2=2$ are not yet confirmed to be topological fundamental groups.

The proof of the results applies the Noether-Horikawa inequality for algebraic surfaces. When a surface has $q=0$, the Noether-Horikawa inequality reads as $K^2\ge 2\chi_h-6$. In addition, for a symplectic $4$-manifold $X$, we could define $\chi_h(X)=\frac{e(X)+\sigma(X)}{4}$. However, even for a simply-connected symplectic $4$-manifold of general type, the inequality does not hold. This gives us space for larger fundamental groups.

\subsection{Invariants for finitely presented groups}\label{invariants}
It is known that any finitely presented group is the fundamental group of some minimal symplectic $4$-manifold by Gompf \cite{Gom}. We could choose these manifolds such that they have $c_1^2=0$ or $c_1^2>0$.

It is then very natural to ask the ``minimal size" of such manifolds with a given fundamental group $G$. Since the first Betti number is determined by the group $G$ as the rank $b_1(G)$, 
natural quantities to describe the size are $b^+$ and $b_2$.


\begin{question}Fix a finitely presented group $G$, what is the minimal $b^+(X)$ (or $b_2(X)$) such that $X$ is a symplectic $4$-manifold with $\pi_1(X)=G$?
\end{question} 

Since the fundamental group $G$ is fixed, asking for minimal $b^+(X)$ is equivalent to asking minimal $\chi_h(X)$ since $\chi_h(X)=\frac{b^+(X)-b_1(X)+1}{2}$. Meanwhile, asking for minimal $b_2(X)$ is equivalent to asking minimal $e(X)=b_2(X)-2b_1(X)+2$.

Let us denote this number by $b^+(G)$ (resp. $b^2(G)$). Similar question could be asked under the assumptions $c_1^2(X)=0$ and $X$ is minimal, in which we call the corresponding numbers $b_0^+(G)$ and $b_0^2(G)$. Especially, these would provide us invariants for finitely presented groups. By definition, $b_0^2(G)\ge b^2(G)$ and $b_0^+(G)\ge b^+(G)$.

By the classification of symplectic $4$-manifolds depending on symplectic Kodaira dimension (see Section \ref{Kodaira}), except for the case when $G$ is a surface group, any minimal symplectic $X$ with $\pi_1(X)=G$ would have $c_1^2(X)\ge 0$. In other word, we have $2e(X)+3\sigma(X)\ge 0$. 
Then $2(2-2b_1(G)+b^++b^-)+3(b^+-b^-)\ge 0$. In other words, $b^+(G)\ge \frac{4b_1(G)-4+b^-(G)}{5}\ge \frac{4b_1(G)-4}{5}$. This gives a lower bound of $b^+(G)$. In general, this is not sharp. 


Gompf's original construction gives us an upper bound. In his examples, $e(X)=12r$ and $\sigma(X)=-8r$ where $r(G)\ge k+l+1$ is a number explicitly determined from a presentation $G = <g_1, \cdots, g_k | r_1, \cdots, r_l>$. It is remarkable that \cite{BKcmh} realizes $r(G)= k+l+1$. Then $b^+(X)=2r(G)+b_1(G)-1=2(k+l)+b_1(G)+1$. Which is an upper bound of our number $b_{(0)}^+(G)$. (It also provides an upper bound of $b_{(0)}^2(G)$ which is $12r(G)+2b_1(G)-2$.) 
If $G$ is not trivial, $2r(G)+b_1(G)-1\ge 3$.  On the other hand, Corollary 6.10 and Proposition 6.4 of \cite{Gom} realizes a free product of $n$ cyclic groups by a symplectic manifold with $c_1^2=0$ and $c_2=48+12m$ whenever $n\le 2m+6$ ($m\ge 0$). Thus $b^+(X)=b_1(G)+1+(2m+6)$, and the last term can be taken to be $n$ rounded up to the nearest even integer (slightly modified when $n\le 4$).
To summarize, for free groups $\mathbb F_n$, construction in \cite{Gom} gives 
\[
b^+(\mathbb F_n)\le\left\{\begin{array}{cl} 
2n+1 &\hbox{ for all even $n\ne 4$},\\
2n+2& \hbox{ for all odd $n$},\\
11 & \hbox{ for $n=4$}.\\
\end{array}\right.
\]

\[
b_2(\mathbb F_n)\le\left\{\begin{array}{cl} 
8n+10 &\hbox{ for all even $n\ne 4$},\\
8n+16& \hbox{ for all odd $n$},\\
54 & \hbox{ for $n=4$}.\\
\end{array}\right.
\]

In the case when $G$ is a surface group, $b^+(G)=1$ and $b^2(G)=2$ if $G$ is nontrivial and $b^2(G)=1$ if $G$ is trivial. They are realized by ruled surfaces.

Moreover, our Question \ref{questionb+=1} is equivalent to: if a group $G$ has a finite group as its abelization, is $b^+(G)=1$? Our Theorem \ref{main} suggests the following more general bound. 

\begin{question} \label{min b+}Is it true that $$b^+(G)\le b_1(G)+1?$$

Moreover, for which group $G$, $b^+(G)<b_1(G)+1$? 
\end{question}

Later in the paper, we verify this first part of the question when $G$ is a free product $({\Large{*}}_{i=1}^n \mathbb{Z}) \Large{*} ({\Large{*}}_{i=1}^k \mathbb Z_{p_i})$, or an abelian group of types $\mathbb Z_p \times \mathbb Z_q$ or $\mathbb Z_p \times \mathbb Z$. In particular, these are much better bounds than those of \cite{Gom} which we summarized above. We notice that for any minimal symplectic $4$-manifold $X$ with $0<c_1^2(X)\le 3c_2(X)$ (i.e. $-2e(X)\le 3\sigma(X)\le e(X)$) and $\pi_1(X)=G$, we have $b^+(X)\ge b_1(X)+1$. 

The second part of Question \ref{min b+} is essentially asking for the fundamental groups of manifolds with $\chi_h\le 0$. Minimal algebraic surfaces with $\chi_h\le 0$ is either ruled surfaces or surface bundles over torus or torus bundles over surface. Luttinger surgery on these manifolds may produce new symplectic manifolds. It is interesting to know the fundamental groups of those manifolds. Especially, manifolds obtained from applying several Luttinger surgeries on torus bundles over torus have special interests in the classification of symplectic $4$-manifolds with $\kappa^s=0$ (c.f. \cite{HoLi}).





There is another problem regarding the geography of manifolds with given fundamental group. 

\begin{question}
For any positive integer $a$, is there a symplectic $4$-manifold $X$ such that $\pi_1(X)=G$ and $b^+(X)=b^+(G)+2a$?
\end{question}

Our construction in the following section shows that we could find symplectic $4$-manifold $X$ such that $b^+(X)=b_1(G)-1+2a$ for any positive integer $a$.

We could similarly define the numbers $b_{alg}^2(G)$ and $b_{alg}^+(G)$ for algebraic surfaces. If $G$ cannot be realized as fundamental group of algebraic surface, e.g. $G=\mathbb Z$ or ${\Large{*}}_{i=1}^{k\ge 2} \mathbb Z_{p_i}$, we define both numbers to be $\infty$. Serre's theorem demonstrates that when $G$ is finite, both numbers are finite.

\section{Luttinger surgery and Arbitrary fundamental groups}
\label{subsec:Luttinger and Kodaira}
Luttinger surgery is introduced in \cite{Lu}. It has been very effective tool recently for constructing exotic smooth structures on $4$-manifolds. In this section, we would construct symplectic $4$-manifolds with arbitrary fundamental groups using Luttinger surgery.

\subsection{Luttinger surgery}
\label{subsec:Luttinger}

Let us briefly review Luttinger surgery. For the details, we refer the reader to \cite{Lu} and \cite{ADK}.

\begin{definition} Let $(X, \omega)$ be a symplectic $4$-manifold, and the torus $\Lambda$ be a Lagrangian submanifold of $X$ (so it has self-intersection $0$). Given a simple co-oriented loop $\lambda$ on $\Lambda$, let $\lambda'$ be a simple loop on $\partial(\nu\Lambda)$ that is parallel to $\lambda$ under the Lagrangian framing. For any integer $m$, the $(\Lambda,\lambda,1/m)$ \emph{Luttinger surgery}\/ on $X$\/ will be 
$X_{\Lambda,\lambda}(1/m) = ( X - \nu(\Lambda) ) \cup_{\phi} (S^1 \times S^1 \times D^2)$,  the $1/m$\/ surgery on $\Lambda$ with respect to $\lambda$ under the Lagrangian framing. Here 
$\phi : S^1 \times S^1 \times \partial D^2 \to \partial(X - \nu(\Lambda))$ denotes a gluing map satisfying $\phi([\partial D^2]) = m[{\lambda'}] + [\mu_{\Lambda}]$ in $H_{1}(\partial(X - \nu(\Lambda))$, where $\mu_{\Lambda}$ is a meridian of $\Lambda$.

\end{definition}

It is  shown in \cite{ADK} that $X_{\Lambda,\lambda}(1/m)$ possesses a symplectic form that restricts to the original symplectic form $\omega$ on $X\setminus\nu\Lambda$. The following lemma is easy to verify and the proof is left as an exercise to the readers.

\begin{lemma} 
\medskip
\medskip

\noindent \item $\pi_1(X_{\Lambda,\lambda}(1/m)) = \pi_1(X- \Lambda)/N(\mu_{\Lambda} \lambda'^m)$.
\smallskip
\noindent \item $\sigma(X)=\sigma(X_{\Lambda,\lambda}(1/m))$ and $e(X)=e(X_{\Lambda,\lambda}(1/m))$.
\end{lemma}

A direct corollary of this lemma is summarized below, which is essentially Proposition 4.4 of \cite{HoLi}. 

\begin{lemma} \label{3case}

The difference $b_1(X_{\Lambda,\lambda}(1/m))-b_1(X)$ is one of $-1, 0 ,1$.  And $$b^+(X_{\Lambda,\lambda}(1/m))-b^+(X)=b_1(X_{\Lambda,\lambda}(1/m))-b_1(X)=\frac{1}{2}(b_2(X_{\Lambda,\lambda}(1/m))-b_2(X)).$$
\end{lemma}


The following result is also due to C.-I. Ho and T.J. Li \cite{HoLi}. 

\begin{theorem}\label{Kod:Luttinger}
The symplectic Kodaira dimension $\kappa^s$ is unchanged under Luttinger surgery.
\end{theorem}

\subsection{Complexity of fundamental group for small manifolds}
Gompf \cite{Gom} shows that any finitely presented group is the fundamental group of some minimal symplectic $4$-manifold. We can choose these manifolds so that $c_1^2=0$. His technique for construction is the symplectic sum, which is often called the {\it fiber sum}.
\smallskip
\smallskip

\noindent {\bf Symplectic sum:} Let $(X_1, \ \omega_{1})$ and $(X_2, \ \omega_{2})$ be closed symplectic $4$-dimensional manifolds. Each of them contains an embedded surface $F_j\subset X_j$ with normal bundle $\nu_j$. Assume that the Euler class of $\nu_i$ satisfy $e(\nu_1) +  e(\nu_2) = 0$. Then for any choice of an orientation reversing $\psi: \nu_1 \cong \nu_2$, the \emph{symplectic  sum} of $X_1$ and $X_2$ along $Y$ is the manifold $(X_1\setminus \nu_1)\cup_{\psi} (X_2\setminus \nu_2)$ and is denoted by $Z=X_1 \#_{\psi} X_2$. The symplectic sum operation provides a natural isotopy class of symplectic structures on $Z$. We have
$$(e, \sigma)(X_1 \#_{\psi} X_2)=(e, \sigma)(X_1)+(e, \sigma)(X_2)+ (4g-4, 0)$$

\smallskip
\smallskip

It is then very natural to ask the ``minimal size" of such manifolds with a given fundamental group $G$. Since the first Betti number $b_1(G)$ is determined by the group $G$ as the rank of $\mathbb R\otimes (G/[G, G])$, a natural quantity to describe the size is $b^+$. In this sense, as discussed in Section \ref{invariants}, Gompf's manifolds all have large size.

In this section, we prove Theorem \ref{main} which provides manifolds with small size and could be viewed as a stratification of Gompf's result according to $b^+$. 

We start with the following result, which verifies Question \ref{min b+} when $G$ is a free product of cyclic groups.

\begin{theorem}\label{freegroup}
For any finitely generated group $G$ of the form $({\Large{*}}_{i=1}^n \mathbb{Z}) \Large{*} ({\Large{*}}_{i=1}^k \mathbb Z_{p_i})$, we have a minimal symplectic $4$-manifold $X(G)$ with $b^+(X(G))=b_1(G)+1$, $\pi_1(X(G))=G$ and $c_1^2(X(G))=0$. Especially, any finitely generated abelian group could be realized as first homology of such symplectic $4$-manifolds.
\end{theorem}

We notice that except for the cyclic group $\mathbb Z_i$, all the other groups are not the fundamental group of K\"ahler surfaces.


\begin{proof}
Our main technique will be $\emph{Luttinger surgery}$. Let us fix integers $p_{i} \geq 0$ and $q_{i} \geq 0$ , where $1 \leq i \leq g$. Let $$Y_{g}(1/p_{1},1/q_{1}, \cdots, 1/p_{g}, 1/q_{g})$$ denote symplectic $4$-manifold obtained by performing the following $2g$ Luttinger surgeries on $\Sigma_{g}\times \mathbb{T}^2$:

\begin{equation}\label{eq: Luttinger surgeries for Y_g(m)} 
\begin{split}
(a_1' \times c', a_1', -1/p_{1}), \ \ (b_1' \times c'', b_1', -1/q_{1}),\\  \nonumber
(a_2' \times c', a_2', -1/p_{2}), \ \ (b_2' \times c'', b_2', -1/q_{2}),\\  \nonumber
 \cdots, \ \ \cdots, \ \ \cdots, \ \  \cdots, \ \ \cdots, \ \ \cdots, \ \  \\ \nonumber
(a_g' \times c', a_g', -1/p_{g}), \ \ (b_g' \times c'', b_g', -1/q_{g}).
\end{split}
\end{equation}

Here, $a_i,b_i$ ($i=1,\cdots, g$) and $c,d$\/ denote the standard generators of $\pi_1(\Sigma_{g})$ and $\pi_1(\mathbb{T}^2)$, respectively. Since all the surgeries above are Luttinger surgeries, $Y_{g}(1/p_{1},1/q_{1},\cdots, 1/p_{g}, 1/q_{g})$ is a minimal symplectic 4-manifold with $\kappa^s=1$. The fundamental group of $Y_{g}(1/p_{1},1/q_{1},\cdots, 1/p_{g}, 1/q_{g})$ is generated by $a_i,b_i$ ($i=1,\cdots, g$) and $c,d$, with the following relations: 

\begin{gather}\label{Luttinger relations for Y_1(m)}
[b_1^{-1},d^{-1}]=a_1^{p_{1}},\ \  [a_1^{-1},d]=b_1^{q_{1}},\ \
[b_2^{-1},d^{-1}]=a_2^{p_{2}},\ \  [a_2^{-1},d]=b_2^{q_{2}},\\ \nonumber
\cdots, \ \ \cdots, \ \ \cdots \  \\ \nonumber
[b_{g}^{-1},d^{-1}]=a_g^{p_{g}},\ \  [a_{g}^{-1},d]=b_g^{q_{g}},\\ \nonumber
[a_1,c]=1,\ \  [b_1,c]=1,\ \ [a_2,c]=1,\ \  [b_2,c]=1,\\ \nonumber
\cdots, [a_g,c]=1,\ \  [b_g,c]=1,\\ \nonumber
[a_1,b_1][a_2,b_2]\cdots[a_g,b_g]=1,\ \ [c,d]=1.
\end{gather}

Let $T \subset Y_{g}(1/p_{1},1/q_{1},\cdots, 1/p_{g}, 1/q_{g})$ denote a genus one symplectic surface that descends from the surface $ pt \times \mathbb{T}^2$ in $\Sigma_{g}\times \mathbb{T}^2$.

Next, we form the symplectic sum of $Y_{g}(1/p_{1},1/q_{1},\cdots, 1/p_{g}, 1/q_{g})$ along the torus $T$ with an elliptic surface $E(1)$. 
\begin{equation*}
X_{g}(p_{1}, q_{1}, \cdots, p_{g}, q_{g})= Y_{g}(1/p_{1},1/q_{1},\cdots, 1/p_{g}, 1/q_{g})\#_{id}(E(1))
\end{equation*}

Choose a base point $x$ of $\pi_1(Y_{g}(1/p_{1},1/q_{1},\cdots, 1/p_{g}, 1/q_{g})$ on $\partial(\nu T)$ such that $\pi_1(Y_{g}(1/p_{1},1/q_{1},\cdots, 1/p_{g}, 1/q_{g})\setminus\nu T,x)$ is normally generated by $a_i,b_i$ ($i=1, \cdots, g$) and $c,d$.  Since the symplectic torus $ pt \times \mathbb{T}^2$ is disjoint from the neighborhoods of $2g$ Lagrangian tori listed above, all the relations in (\ref{Luttinger relations for Y_1(m)}) continue to hold in $\pi_1(Y_{g}(1/p_{1},1/q_{1},\cdots, 1/p_{g}, 1/q_{g})\setminus\nu T)$ except the relation $[a_1,b_1][a_2,b_2]\cdots[a_g,b_g]=1$. The surface relation $[a_1,b_1][a_2,b_2]\cdots[a_g,b_g]$ is no longer trivial. In $\pi_1(Y_{g}(1/p_{1},1/q_{1},\cdots, 1/p_{g}, 1/q_{g}) \setminus\nu T)$, it represents a meridian of $T$.

Since $\pi_1(E(1) \setminus (\nu(T)) = 1$, after the fiber sum we have $c = d = 1$ in the fundamental group of $X_{g}(p_{1}, q_{1}, \cdots, p_{g}, q_{g})$. As a consequence of this, we obtain the following presentation for the fundamental group of $X_{g}(p_{1}, q_{1}, \cdots, p_{g}, q_{g})$:

\begin{gather}\label{Relations}
a_1^{p_1} = 1, \nonumber  \\ 
b_1^{q_1} = 1, \nonumber \\
\cdots, \nonumber \\
\cdots, \nonumber \\
a_g^{p_g} = 1, \nonumber \\ 
b_g^{q_g} = 1, \nonumber \\
\Pi_{j=1}^{g} [a_j, b_j] = 1. \nonumber \\
\nonumber
\end{gather}

By setting  $p_1 = q_1 = \cdots = p_g = q_g = 1$, we obtain a symplectic $4$-manifold $X_{g}(1, 1, \cdots, 1, 1)$ with trivial fundamental group. The symplectic $4$-manifold $X_{g}(1,0, \cdots 1, 0)$ has the fundamental group $F_{g}$, a free group of rank $g$. In general, the symplectic $4$-manifold $X_{g, k}(1,0, \cdots 1, 0, 1, p_1, \cdots, 1, p_k)$ has the fundamental group $({\Large{*}}_{i=1}^n \mathbb{Z}) \Large{*} ({\Large{*}}_{i=1}^k \mathbb Z_{p_i})$. The manifold $X_{g, k}(1,0, \cdots 1, 0, 1, p_1, \cdots, 1, p_k)$ has Euler number $e=12$ and signature $\sigma=-8$. Since the first Betti number $b_1=g$, we have $b^+=g+1$ and $b^-=g+9$.
\end{proof}



Now, we are ready to prove Theorem \ref{main}. For the convenience of readers, let us rephrase it in the below.
\begin{theorem} \label{stratGompf}
For any finitely presented group $G$ which is dual finite torsion, there is a finitely presented group $G'$ with surjective group homomorphism $G'\twoheadrightarrow G$, such that there is a minimal symplectic $4$-manifolds $X(G)$ with $b^+(X(G))=b_1(G)+1$, $\pi_1(X(G))=G'$ and $c_1^2(X(G))=0$. Moreover, $b_1(G)=b_1(G')$.

\end{theorem}

\begin{proof}

Now, let $G = \langle x_{1}, \ \cdots, \ x_{k} \ | \ l_{1}, \ \cdots, \ l_{m} \rangle$ be any finitely presented group whose Abelianization is denoted by $AG$. To construct manifolds with fundamental group $G$, we use similar strategy as for Proposition \ref{freegroup}. 

A new ingredient is the {\it bridge move} operation, as illustrated in Figure \ref{fig:bridge1}. It is used to resolve the self-intersection points of an immersed loop on surface $\Sigma$. Namely, for each self-intersection point $x$ with two segments $l$ and $m$ intersecting locally, we increase the genus of $\Sigma$ by adding a ``bridge'' with $l$ going across the bridge and $m$ going under the bridge. This operation of resolving self-intersection points is introduced in \cite{ABKP} (see also \cite{Kork}).

\begin{figure}
\centering
\includegraphics[scale=.43]{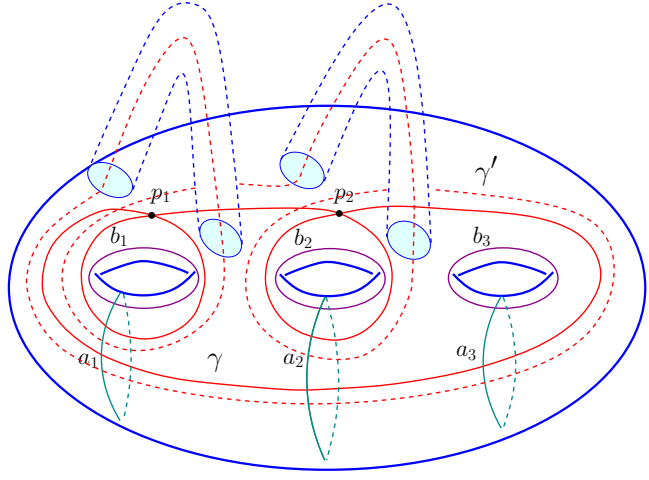}
\caption{Bridge move on $\Sigma_3$ for the relation $b_{1}b_{2}b_{3}b_{2}b_{1}$}
\label{fig:bridge1}
\end{figure}

With this operation in mind, we prove Theorem \ref{stratGompf}.
For each word $l_i$ with $i=1, \cdots, m$, let $\gamma_i$ be a smooth immersed oriented circle on $F=\Sigma_{k}$ representing the word $l_i$ with generators $x_j$ replaced by $b_j$. Possibly after perturbation, we could choose the loops $\gamma_i$ such that only two segments intersect locally at each self-intersection point.

Now, for each self intersection of $\gamma_i$, we resolve it by adding a bridge over it which increase the genus of $F$ by one (see Figure \ref{fig:bridge1}). We will denote the modified curves by $\gamma_i'$. After finitely many such ``cross bridge'' moves, $F$ is changed to a surface $F'$ of genus $k'>k$ with $\gamma_i'$ all embedded curves. We want each $\gamma_i'$ passes through odd number of bridges. If not, we add one more bridge along any (smooth) point of the curve $\gamma_i'$. Notice that each bridge would have only one $\gamma_i'$ passing through. For each increased genus, we also have a pair of new generators $a_g$ and $b_g$. We order these bridges as the following. Let $\gamma'_i$ pass through the handle $a_g\times b_g$ and $\gamma'_j$ pass through the handle  $a_{g'}\times b_{g'}$. If $i>j$, then $g>g'$. Furthermore, we assume that the first bridge $\gamma_i'$ passes through is $a_{g_i}\times b_{g_i}$.

We perform Luttinger surgeries
$$(a_g' \times c', a_g', -1), \ \ (b_g' \times c'', b_g', (-1)^{c(g)-1}), \ \ k+1\le g\le k'.$$
Here $c(g)=\min_{1\le i\le m}\{g-g_i+1| g-g_i\ge 0\}$.

As for free groups, we again take the same $2k$ Lagrangian tori in $\Sigma_{k'}\times \mathbb{T}^2$ corresponding to $a_g$ and $b_g$ with $1\le g\le k$, and choose $p_1=\cdots =p_k=1$, $q_1=\cdots =q_k=0$. 

In addition to applying these $2k'$ Luttinger surgeries to these standard tori, we apply $m$ more surgeries:

$$(\gamma_1' \times c''', \gamma_1', -1), \cdots,  (\gamma_m' \times c''', \gamma_m', -1).$$


Let us denote the manifold after the $2k'+m$ surgeries by $Y(G)$. Then we apply symplectic sum of $Y(G)$ along the torus $T$ descending from $ pt \times \mathbb{T}^2$ with $E(1)$, and denote the resulting symplectic manifold by $X(G)$.

The ``dual tori'' $a_{g}\times d$ for $g>k$ is now twice punctured (corresponding to surgeries $(b_g' \times c'', b_g', (-1)^{c(g)-1})$ and $(\gamma_{k_g}' \times c''', \gamma_{k_g}', -1)$), and would produce relation $[a_{g}^{-1},d]=\gamma'_{k_g}b_g^{(-1)^{c(g)}}$ in the fundamental group of $Y(G)$. Here $\gamma'_{k_g}$ denotes the $\gamma_i'$ passing through the bridge $a_g\times b_g$.



After fiber summing, $c=d=1$. These relations would become 
\begin{equation}\label{gammarelation}
\gamma'_{k_g}b_g^{(-1)^{c(g)}}=1, g>k.
\end{equation}

Let us fix $k_g=i$ in the following. After abelianization, $\gamma_i'$ could be written as $\gamma_ib_{g_i}\cdots b_{g_{i+1}-1}$. Since each $\gamma_i'$ passes through odd number of bridges, $g_{i+1}-g_i$ is an odd number. Now, the relations \eqref{gammarelation} for $k_g=i$ read as  
\begin{equation}\label{i-relation}
b_{g_i}=\gamma_i', \quad b_{g_i+1}=\gamma_i'^{-1}, \cdots, \quad b_{g_{i+1}-1}=\gamma_i'^{(-1)^{g_{i+1}-g_i-1}}=\gamma_i'.
\end{equation}
Hence we have 
\begin{equation}\label{product}
b_{g_i}\cdots b_{g_{i+1}-1}=\gamma_i'. 
\end{equation}
Recall the relation $\gamma_i'b_{g_i}^{-1}=1$ becomes $\gamma_ib_{g_i}\cdots b_{g_{i+1}-1}\cdot b_{g_i}^{-1}=1$ after abelianization. After replacing \eqref{product} into this new expression, we have $\gamma_i\gamma_i'b_{g_i}^{-1}=1$. Since $\gamma_i'b_{g_i}^{-1}=1$ by \eqref{gammarelation}, we finally have 
 relations $\gamma_i=1$ in the abelianization of $\pi_1(X(G))$. By \eqref{i-relation}, for each relator $l_i$ in the presentation of $G$, we introduce only one more free generator $b_{g_i}=\gamma_i'$.
 

The other relations in $\pi_1(Y(G))$ are almost the ones in (\ref{Luttinger relations for Y_1(m)}) (with $g=k$). The only differences are the commutators $[a_i^{-1}, d]$ would be equal to $ P_i(\gamma_1', \cdots, \gamma_m')$ since $q_i$ are set to be zero, where $P_i(\gamma_1', \cdots, \gamma_m')$ is a monomial of the words $\gamma_1', \cdots, \gamma_m'$. There relations are determined as following. We write the relators $l_i$, after abelianization, as $x_1^{m_{1i}}\cdots x_k^{m_{ki}}$. Then $P_i(\gamma_1', \cdots, \gamma_m')=\gamma_1'^{m_{i1}}\cdots \gamma_m'^{m_{im}}$. Hence if the matrix $(m_{ij})_{1\le i\le k, 1\le j\le m}$ has rank $m$, we have $\gamma_1', \cdots, \gamma_m'$ are all of finite order.


To summarize, we have $b_1, \cdots, b_k$ and $\gamma_1', \cdots, \gamma_m'$ as the generators in the abelianization of $\pi_1(X(G))$. In addition to the commutator relations, we have $\gamma_1=\cdots=\gamma_m=1$. All $\gamma_i$ are words of $b_1, \cdots, b_k$. Moreover, we also have $\gamma_1', \cdots, \gamma_m'$ are of finite order.
In other words, we have constructed symplectic $4$-manifold $X(G)$ with the first homology $H_1(X(G), \mathbb Q)=AG\otimes \mathbb Q$, where $AG$ is the abelianization of $G$. 
Since Luttinger surgery would preserve minimality, $Y(G)$ is minimal. From Usher's theorem \cite{U1}, $X(G)$ is minimal.

Let us denote $G'=\pi_1(X(G))$. As discussed above, $b_1, \cdots, b_{k'}$ are generators of this group. The surjective group homomorphism $G'\twoheadrightarrow G$ is the natural one: sending $b_1, \cdots, b_k$ to $x_1, \cdots, x_k$ and $b_{k+1}, \cdots, b_{k'}$ to $1$ (remember the relations in \eqref{i-relation}).

Finally, let us calculate their topological quantities. Since the Luttinger surgeries keep the Euler number and signature, $e(Y(G))=\sigma(Y(G))=0$. After fiber summing, $e(X(G))=12$ and $\sigma(X(G))=-8$, thus $c_1^2=2e+3\sigma=0$. 
Finally, since $b_1(X(G))=b_1(G)$, we have $b^+(X(G))=b_1(G)+1$ and $b^-(X(G))=b_1(G)+9$. This completes the proof of Theorem \ref{stratGompf}. 
\end{proof}


When varying $G$ with the same $\mathbb Q$-abelianization $AG\otimes \mathbb Q$, we would have infinitely many non-K\"ahler symplectic $4$-manifolds whose first $\mathbb Q$-homology groups are $AG\otimes \mathbb Q$.

If we fiber sum $X(G)$ along the torus $T$ descending from $ pt \times \mathbb{T}^2$ with $E(n)$ along with torus fiber, we would obtain symplectic $4$-manifolds with $H_1^{\mathbb Q}=AG\otimes \mathbb Q$ and $\chi_h=n$, i.e. $b^+=b_1+2n-1$.

\begin{remark}\label{moregenerators}
If $G = \langle x_{1}, \ \cdots, \ x_{k}, \ x_{k+1}, \cdots, \ x_{k+m} \ | \ l_{1}, \ \cdots, \ l_{m} \rangle$, where $l_i$ are words of  $x_{1}, \ \cdots, \ x_{k}$ only, then we have symplectic $4$-manifold $X(G)$ with $b^+(X(G))=b_1(G)+1$, $\pi_1(X(G))=G$ and $c_1^2(X(G))=0$.

To achieve this, we modify the construction for Theorem \ref{stratGompf} by starting with a genus $k$ surface with generators $a_1, b_1, \cdots, a_k, b_k$ and set the coefficients of Luttinger surgeries for $(b_{g_i}' \times c'', b_{g_i}')$ to be $0$.
\end{remark}

This following remark is essentially due to Jun Yu.

\begin{remark}\label{orbifoldgroup} As was mentioned before, the fundamental groups of an elliptic surface with $b^+=\chi_h=1$ are groups $E_{p_1, p_2, \cdots, p_k}=\{x_1, x_2, \cdots, x_k|x_1^{p_1}=x_2^{p_2}=\cdots=x_k^{p_k}=x_1x_2\cdots x_k=1\}$. Note that the possible quotients of all $E_{p_1, p_2, \cdots, p_k}$ and that of free product of cyclic groups are the same. Clearly, there is a surjection ${\Large{*}}_{i=1}^k \mathbb Z_{p_i}\rightarrow E_{p_1, p_2, \cdots, p_k}$. On the other hand,  $E_{p_1, p_2, \cdots, p_k, p_k, \cdots, p_1}$ surjects to ${\Large{*}}_{i=1}^k \mathbb Z_{p_i}$. This can be seen by adding the relations $x_{2k-i}=x_i^{-1}$, and then $x_1x_2\cdots x_{2k}=1$ holds automatically.

However, there are many $\mathbb Q$-perfect groups which are not the quotient of any $E_{p_1, p_2, \cdots, p_k}$ or equivalently of ${\Large{*}}_{i=1}^k \mathbb Z_{p_i}$. In fact, there are even torsion free (i.e. no finite order elements) $\mathbb Q$-perfect finitely presented groups. The congruence subgroups $K_{n,m}$ are the group of $n\times n$ integral matrices which are congruent to identity modulo $m$. These groups are finitely presented and with finite (but non-trivial) abelianization. When $m=p^{k}>>1$, they are torsion free. For example, $K_{3,4}$ is torsion free. Moreover, a finitely presented torsion-free simple (thus perfect) group is constructed in \cite{Rat}.

It is also worth mentioning that many interesting groups (e.g. all finite groups, free products of cyclic groups and Thompson groups $T$ and $V$) are quotients of certain $E_{p_1, p_2, \cdots, p_k}$.
\end{remark}

\subsection{Construction of symplectic $4$-manifolds with  $c_{1}^2 \geq 1$ and $H_1^{\mathbb Q} = AG\otimes \mathbb Q$} \label{sec:c}

Our construction strategy for Proposition \ref{freegroup} and Theorem \ref{stratGompf} could be applied to construct symplectic $4$-manifolds with $c_1^2\ge 1$ as well. In this section, we sketch the construction for this case.


To produce such examples, we can proceed as follows. Let us fix integers $n\geq 2$, $p_{i} \geq 0$ and $q_{i} \geq 0$ , where $1 \leq i \leq n$. Let $Y_{n}(1/p_{1},1/q_{1}, \cdots, 1/p_{n}, 1/q_{n})$ denote symplectic $4$-manifold obtained by performing $2n + 4$ Luttinger surgeries on $\Sigma_{n}\times \Sigma_{2}$ (\cite{FPS, AP}). These $2n+4$ surgeries consist of the following $8$ surgeries

\begin{eqnarray}\label{first 8 Luttinger surgeries}
&&(a_1' \times c_1', a_1', -1), \ \ \nonumber (b_1' \times c_1'', b_1', -1), \\ \nonumber
&&(a_2' \times c_2', a_2', -1), \ \ (b_2' \times c_2'', b_2', -1),\\ \nonumber
&&(a_2' \times c_1', c_1', +1/p_1), \ \ (a_2'' \times d_1', d_1', +1/q_1),\\ \nonumber
&&(a_1' \times c_2', c_2', +1/p_2), \ \ (a_1'' \times d_2', d_2', +1/q_2),
\end{eqnarray}
together with the following $2(n-2)$ additional Luttinger surgeries
\begin{gather*}
(b_1'\times c_3', c_3',  -1/p_3), \ \ 
(b_2'\times d_3', d_3', -1/q_3), \\  \cdots  ,\\
(b_1'\times c_n', c_n',  -1/p_n), \ \
(b_2'\times d_n', d_n', -1/q_n).
\end{gather*}
Here, $a_i,b_i$ ($i=1,2$) and $c_j,d_j$ ($j=1,\cdots,n$) denote the standard loops that generate $\pi_1(\Sigma_2)$ and $\pi_1(\Sigma_n)$, respectively. See Figure~\ref{fig:lagrangian-pair} for a typical Lagrangian tori along which the surgeries are performed. The Figure~\ref{fig:lagrangian-pair}, which we borrowed from \cite{AP2} (with a minor modification), illustrates a typical Lagrangian tori along which our Luttinger surgeries performed.  

\begin{figure}[ht]
\begin{center}
\includegraphics[scale=.89]{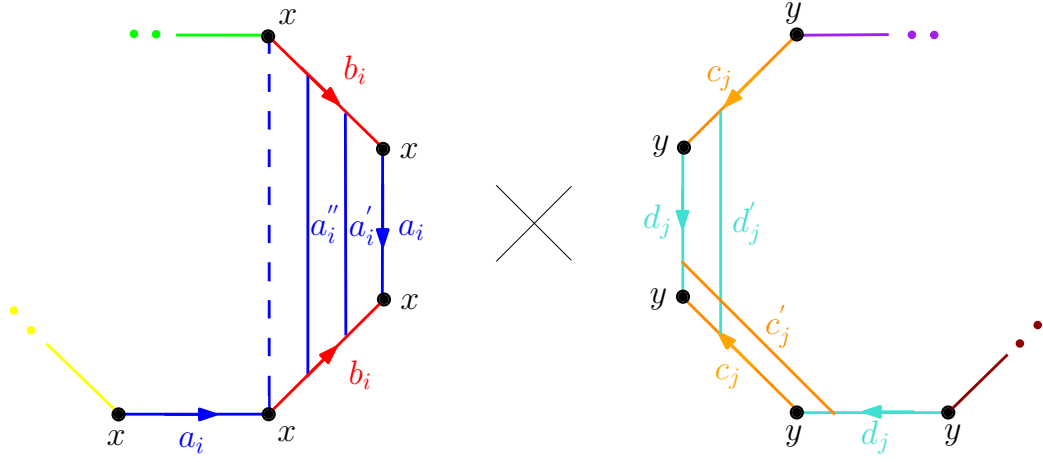}
\caption{Lagrangian tori $a_i'\times c_j'$ and $a_i''\times d_j'$}
\label{fig:lagrangian-pair}
\end{center}
\end{figure}

The Euler characteristic of $Y_{n}(1/p_{1},1/q_{1}, \cdots, 1/p_{n}, 1/q_{n})$ is $4n-4$ and its signature is $0$. The fundamental group $\pi_1(Y_{n}(1/p_{1},1/q_{1}, \cdots, 1/p_{n}, 1/q_{n}))$ is generated by $a_i,b_i,c_j,d_j$ ($i=1,2$ and $j=1,\cdots,n$) and the following relations hold in $\pi_1(Y_{n}(1/p_{1},1/q_{1}, \cdots, 1/p_{n}, 1/q_{n}))$:  

\begin{gather}\label{Luttinger relations}
[b_1^{-1},d_1^{-1}]=a_1,\ \  [a_1^{-1},d_1]=b_1,\ \  [b_2^{-1},d_2^{-1}]=a_2,\ \  [a_2^{-1},d_2]=b_2,\\ \nonumber
[d_1^{-1},b_2^{-1}]=c_1^{p_1},\ \ [c_1^{-1},b_2]=d_1^{q_1},\ \ [d^{-1}_2,b^{-1}_1]=c_2^{p_2},\ \ [c_2^{-1},b_1]=d_2^{q_2},\\ \nonumber
 [a_1,c_1]=1, \ \ [a_1,c_2]=1,\ \  [a_1,d_2]=1,\ \ [b_1,c_1]=1,\\ \nonumber
[a_2,c_1]=1, \ \ [a_2,c_2]=1,\ \  [a_2,d_1]=1,\ \ [b_2,c_2]=1,\\ \nonumber
[a_1,b_1][a_2,b_2]=1,\ \ \prod_{j=1}^n[c_j,d_j]=1,\\ \nonumber
[a_1^{-1},d_3^{-1}]=c_3^{p_3}, \ \ [a_2^{-1} ,c_3^{-1}] =d_3^{q_3}, \  \cdots, \ 
[a_1^{-1},d_n^{-1}]=c_n^{p_n}, \ \ [a_2^{-1} ,c_n^{-1}] =d_n^{q_n},\\ \nonumber
[b_1,c_3]=1,\ \  [b_2,d_3]=1,\ \cdots, \
[b_1,c_n]=1,\ \ [b_2,d_n]=1.
\end{gather}

The surfaces $\Sigma_2\times\{{\rm pt}\}$ and $\{{\rm pt}\}\times \Sigma_n$ in $\Sigma_2\times\Sigma_n$ descend to surfaces in $Y_{n}(1/p_{1},1/q_{1}, \cdots, 1/p_{n}, 1/q_{n})$. They are still  symplectic submanifolds in $Y_{n}(1/p_{1},1/q_{1}, \cdots, 1/p_{n}, 1/q_{n})$. We will denote their images by $\Sigma_2$ and $\Sigma_n$. Note that $[\Sigma_2]^2=[\Sigma_n]^2=0$ and $[\Sigma_2]\cdot[\Sigma_n]=1$. 

Another building block we would make use is $Z''(1,1)$ of \cite{AP}. If we denote $\alpha_i$, $i=1, \cdots, 4$ to be the standard generators in $\pi_1(\mathbb T^4)$, it is constructed from $\mathbb T^4\# \CPb$ by performing the following two Luttinger surgeries:
\begin{equation}\label{eq: Luttinger surgeries in T^4}
(\alpha_2'\times\alpha_3',\alpha_3',-1), \ \ (\alpha_2''\times\alpha_4',\alpha_4',-1).  
\end{equation}

There is a genus two symplectic surface $\bar \Sigma_2$ in $\mathbb T^4\# \CPb$ with homology class $[\alpha_1\times \alpha_2]+2[\alpha_3\times \alpha_4]-2[E]$. The inclusion $\bar \Sigma_2\rightarrow \mathbb T^4\# \CPb$ maps the generators of $\pi_1(\bar \Sigma_2)$ as:
\begin{equation}\label{eq: embedding of Sigma_2'}
\bar{a}_1  \mapsto  \alpha_1, \ \ 
\bar{b}_1  \mapsto  \alpha_2, \ \ 
\bar{a}_2  \mapsto  \alpha_3^2, \ \ 
\bar{b}_2  \mapsto  \alpha_4.
\end{equation}
It is still a symplectic submanifold of $Z''(1,1)$ with self-intersection zero.

We have the following theorem, which is Theorem 6 in \cite{AP}.

\begin{theorem}\label{Z''}
There exists a nonnegative integer\/ $s$ such that\/ $\pi_1(Z''(1, 1)\setminus\nu\bar{\Sigma}_2)$ is a quotient of the following finitely presented group
\begin{eqnarray}\label{pi_1(Sigma_2' complement in Z''(m))}
\langle
\alpha_1,\alpha_2,\alpha_3,\alpha_4, g_1, \cdots, g_s 
&\mid& \alpha_3=[\alpha_1^{-1},\alpha_4^{-1}],\, \alpha_4=[\alpha_1,\alpha_3^{-1}],\\
&&[\alpha_2,\alpha_3]=[\alpha_2,\alpha_4]=1
\rangle.  \nonumber
\end{eqnarray}
In\/ $\pi_1(Z''(1, 1)\setminus\nu\bar{\Sigma}_2)$, 
the images of\/ $g_1,\cdots,g_s$ are elements of the subgroup normally generated by the image of\/ $[\alpha_3,\alpha_4]$.  
\end{theorem}

We form the symplectic sum of $Y_{n}(1/p_{1},1/q_{1},\cdots, 1/p_{n}, 1/q_{n})$ along the genus two surface with $Z''(1,1)$. 

\begin{equation*}
X_{n}(p_{1}, q_{1}, \cdots, p_{n}, q_{n})= Y_{n}(1/p_{1},1/q_{1},\cdots, 1/p_{n}, 1/q_{n})\#_{id}Z''(1,1)
\end{equation*}

We claim the fundamental group of $X_{n}(1, 1, 1, 1, 1, 0 \cdots, 1, 0)$ is free group $\mathbb F_{n-2}$ with $n-2$ generators.

From Seifert-Van Kampen theorem,
the fundamental group is a quotient of the following group:  
\begin{equation}\label{pi_1(X_n(m))}
\frac{\pi_1(Y_{n}(1, 1, 1, 1, 1, 1/0, \cdots, 1, 1/0)\setminus\nu\Sigma_{2})\ast
\pi_1(Z''(1,1)\setminus\nu\bar{\Sigma}_2)}{\langle a_1=\alpha_1,\,
b_1=\alpha_2,\, a_2=a_3^2,\, 
b_2=\alpha_4,\, 
\mu(\Sigma_{2})=\mu(\bar{\Sigma}_2)^{-1} \rangle}.
\end{equation}

All the relations in (\ref{Luttinger relations}) continue to hold in (\ref{pi_1(X_n(m))}) except possibly for
$\prod_{j=1}^n[c_j,d_j]=1$.  This product may no longer be trivial and now represents a meridian of $\Sigma_{2}$.  

In (\ref{pi_1(X_n(m))}), we have $a_1=[b_1^{-1},d_1^{-1}]=[b_1^{-1},[c_1^{-1},b_2]^{-1}]=[b_1^{-1},[b_2,c_1^{-1}]]$.  
Since $[\alpha_2,\alpha_4]=1$ in (\ref{pi_1(Sigma_2' complement in Z''(m))}), we conclude that $[b_1,b_2]=1$.  Since we also have $[b_1,c_1]=1$ from (\ref{Luttinger relations}), we easily deduce that $a_1=1$. From $a_1=1$ and first eight relations in \eqref{Luttinger relations}, we have $a_i$, $b_i$, $c_i$, $d_i$ are all trivial when $i=1$, $2$. We could also kill all $c_i$ for $i=3, \cdots, n$. Consequently, all the generators $\alpha_i$ would be all killed. Especially, $\alpha_3=[\alpha_1^{-1},\alpha_4^{-1}]=[a_1^{-1},b_2^{-1}]=1$. Finally, since $[\alpha_3,\alpha_4]$  is trivial and hence the generators $g_1,\cdots, g_s$ die as well. In conclusion, the fundamental group is a free group generated by $d_3, \cdots, d_n$.

In general, we apply bridge moves as we did in the proof of Theorem \ref{stratGompf}. Let $G = \langle x_{1}, \ \cdots, \ x_{k} \ | \ l_{1}, \ \cdots, \ l_{m} \rangle$ whose $\mathbb Q$-abelianization equals $AG\otimes \mathbb Q$. We represent $x_1, \cdots, x_k$ by $d_3, \cdots, d_{k+2}$. For any relation $l_i$ in a given presentation of $G$, we find a immersed loop $\gamma_i$ generated by $d_i$ with $i=3, \cdots, k+2$. After performing bridge moves, the self-intersection points of $\gamma_i$ are resolve with the cost of increasing genus $k+2$ to $k'$. Similarly, we order the bridges corresponding to the order of $\gamma_i'$. In addition to the $2(k+2)+4$ Luttinger surgeries in the beginning of \ref{sec:c}, we perform Luttinger surgeries to generators $c_g$, $d_g$ of added bridges:

$$(c_g' \times b_1', c_g', -1), \ \ (d_g' \times b_2', d_g', (-1)^{c(g)-1}), \ \  k+3\le g \le k',$$
and to the embedded loops $\gamma_i'$:

$$(\gamma_1' \times b_2', \gamma_1', -1), \cdots,  (\gamma_m' \times b_2', \gamma_m', -1).$$

We call this manifold $Y^+(G)$. Similarly, we apply symplectic sum of $Y^+(G)$ along the genus two surface $\Sigma_2$ descending from $\Sigma_2\times pt$ with $Z''(1,1)$, and denote the resulting symplectic manifold by $X^+(G)$.




Exactly the same argument as in Theorem \ref{stratGompf}, the abelianization of $\pi_1(X(G)$ is $AG\otimes \mathbb Q$ if the presentation $G= \langle x_{1}, \ \cdots, \ x_{k} \ | \ l_{1}, \ \cdots, \ l_{m} \rangle$ is dual finite torsion.
Moreover, by Usher's theorem \cite{U1}, $X^+(G)$ is minimal.

Our manifold $X^+(G)$ has  $e(X^+(G))=4k'+1$ and $\sigma(X^+(G))=-1$, thus $c_1^2=8k'-1$ and $\chi_h=k'$. However, the number $k'$ is determined by the given presentation of $G$. It would be interesting to know smaller examples and $c_1^2>0$.

The following remark was pointed out to us by the referee.

\begin{remark} Notice that each symplectic $4$-manifold constructed in the previous section contains a $\pi_{1}$-trivial symplectic torus with self-intersection zero. This torus descends to our manifold from the surface $pt \times \mathbb{T}^2$ of $F' \times \mathbb{T}^{2}$. Thus, one should be able to realize any $c_{1}^{2} \geq 0$ for each corresponding group, with additional control on the geography, by Theorem 6.2 in \cite{Gom}, combined with its subsequent Remark 2, and Usher's Theorem on symplectic minimality in \cite{U1}. 
\end{remark}


\subsection{Small symplectic $4$-manifolds with $1 \leq c_1^2 \leq 7$}

The construction in the previous section provides examples with $c_1^2>0$ but with large $b^+$. Our goal in this section is to provide examples with minimal $b^+$ for some abelian groups. Our examples with $c_1^2=1, 2$ and $b^+=1$ are interesting to compare with Theorem \ref{reid}, which implies that the algebraic fundamental group of minimal algebraic surfaces with $c_1^2=1$ (resp. $c_1^2=2$) and $b^+=1$ could only be cyclic groups of order at most five (resp. groups of order at most nine). 

\subsubsection{Small symplectic $4$-manifolds with $c_1^2=1$}
\

To construct our first example, we use the genus $2$ symplectic surfaces of self-intersection $0$ in $(\mathbb{T}^2 \times \mathbb{S}^2)\#4\,\CPb$ and $(\mathbb{T}^2 \times \mathbb{S}^2)\#3\,\CPb$.

First, we consider $\mathbb{T}^2 \times \mathbb{S}^2$ and take the union of three symplectic surfaces $(\mathbb T^2\times  \{s_1\} )\cup (\{t\}\times \mathbb S^2)\cup (\mathbb T^2\times \{s_2\})$ in it. Resolve the two double points symplectically, we have a genus two symplectic surface in $\mathbb{T}^2 \times \mathbb{S}^2$ with self-intersection four.  Blow up four times, we obtain a symplectic genus two surface $\Sigma_2$ of self-intersection zero in $X_1=(\mathbb{T}^2 \times \mathbb{S}^2)\#4\,\CPb$.


Next, we consider $\mathbb{T}^2 \times \mathbb{S}^2$ with a standard product symplectic form, and let $T_1 = \mathbb{T}^2 \times pt$, $T_2 = 2(\mathbb{T}^2 \times pt)$, and $S = pt \times \mathbb{S}^2$ be the symplectic configuration of curves show in Figure~\ref{fig:bc}. 

\begin{figure}[H]
\begin{center}
\includegraphics[scale=.53]{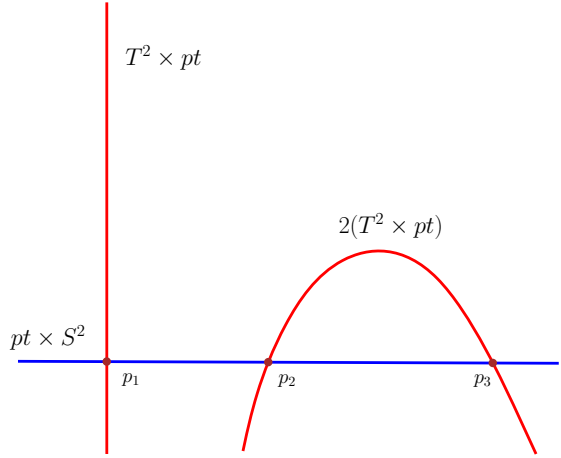}
\caption{Configuration of curves in $\mathbb{T}^2 \times \mathbb{S}^2$}
\label{fig:bc}
\end{center}
\end{figure}

Symplectically resolve the intersection points $p_1$, and $p_2$, and symplectically blow up $p_3$ to obtain a genus two symplectic surface of self-intersection two in $(\mathbb{T}^2 \times \mathbb{S}^2) \#\,\CPb$. Next, we symplectically blow up this surface twice to obtain a genus $2$ symplectic surface $\Sigma_2'$ of self-intersection $0$ in $X_2=\mathbb{T}^2 \times \mathbb{S}^2 \#3\,\CPb$. Let us denote the standard generators of $\pi_{1}(\Sigma_2')$  under the inclusion by $c_{2}$, $d_{2}$, $a_{2}$, $b_{2}$. The inclusion induced homomorphism from $\pi_1(\Sigma'_2)$ into $\pi_1(X_2)=\langle x,y| [x, y]=1\rangle$ as:
$$c_2\mapsto x, \ \  d_2\mapsto y, \ \ a_2\mapsto x^{-2}, \ \ b_2\mapsto y^{-1}.$$

Similarly, let us denote the standard generators of $\pi_1(\Sigma_2)$ in $X_1$ by $a_1$ and $b_1$, and $c_1$ and $d_1$ respectively. The inclusion induced homomorphism from $\pi_1(\Sigma'_2)$ into $\pi_1(X_2)=\langle x',y'| [x', y']=1\rangle$ as:
$$a_1\mapsto x', \ \  b_1\mapsto y', \ \ c_1\mapsto x'^{-1}, \ \ d_1\mapsto y'^{-1}.$$ 
 
The normal circles $\mu$ and $\mu'$ of $\Sigma_2$ and $\Sigma_2'$, in $X_1 \setminus \Sigma_2$ and $X_2 \setminus \Sigma_2'$, can be deformed using the exceptional spheres coming from the blow ups. Hence the fundamental group of the complements of $\Sigma_2$ in $X_1$ and $\Sigma_2'$ in $X_2$ are both isomorphic to $\mathbb{Z}^2$, generated by $a_1$ and $b_1$, and $c_1$ and $d_1$ respectively.

We take the symplectic fiber sum of $X_1 = \mathbb{T}^2 \times \mathbb{S}^2 \#4\,\CPb$ and $X_2  = \mathbb{T}^2 \times \mathbb{S}^2 \#3\,\CPb$ along the surfaces $\Sigma_2$ and $\Sigma_2'$, via a diffeomorphism that extends the orientation preserving diffeomorphism from $\Sigma_2'$ (with $a_2, b_2, c_2, d_2$ generates $H_1(\Sigma_2')$) to $\Sigma_2$ (with $a_1, b_1, c_1, d_1$ generates $H_1(\Sigma_2)$) described by:
$$a_1\mapsto a_2^ab_2, \ \ b_1 \mapsto a_2^{ad-1}b_2^d, \ \ c_1\mapsto c_2^{2a+p}d_2, \ \ d_1\mapsto c_2^{2ad-2}d_2^{d-q}.$$
Here we assume $\gcd(p, 2q)=1$ and the pair $(a, d)$ solves $dp-2aq=pq-1$. One can check that all the intersection pairings are preserved, hence there is an orientation preserving diffeomorphism from $\Sigma_2'$ to $\Sigma_2$.

By Seifert-Van Kampen theorem, the fundamental group of the resulting manifold $X'$ can be seen to be generated by $x$, $y$, $x'$ and $y'$, which all commute with each other and satisfying $xx' = 1, yy' = 1$, with the relations $$x^{-2a}y^{-1}x^{2a+p}y=1, \quad x^{2-2ad}y^{-d}x^{2ad-2}y^{d-q}=1.$$ The latter two relations are equivalent to $x^p=1, y^q=1$.  Thus $\pi_1(X)$ is isomorphic to $\mathbb Z_{|p|} \times  \mathbb Z_{|q|}$. It is easy to check that $e(X) = 11$ and $\sigma(X) = -7$. Usher's Theorem guarantees that $X$ is minimal. Notice $pq\ne 0$, hence all the manifolds here have $b^+=1$. 

Hence, we have 

\begin{theorem}\label{K2=1}
For any $p, q\ge 1$ with $\gcd(p, 2q)=1$, there exists minimal symplectic $4$-manifold $X_{p, q}$ such that  
$$\pi_1(X_{p,q})=\mathbb Z_p\times \mathbb Z_q, \quad b^+(X_{p,q})=1, \quad c_1^2(X_{p,q})=1.$$
\end{theorem}

Especially, all the cyclic groups could realized by letting $p=1$.

\smallskip

The following remark has been mostly communicated to us by the referee. 

\begin{remark} The construction above also works with $X_2$ replaced by $\mathbb{T}^2 \times \mathbb{S}^2 \#2\,\CPb$, where the genus two surface is obtained using the tori $T_{1} = \mathbb{T}^2 \times pt$, $T_3 = 3(\mathbb{T}^2 \times pt)$, and sphere $S = pt \times \mathbb{S}^2$ in $\mathbb{T}^2 \times \mathbb{S}^2$. We symplectically blow up two intersection points $p_3$, $p_4$ of $T_3$ and $S$, and symplectically resolve the remaining two intersection points $p_1$ of $T_{1}$ and $S$, and $p_2$ of $T_{2}$ and $S$ to obtain a genus two symplectic surface of self-intersection zero in $\mathbb{T}^2 \times \mathbb{S}^2 \#2\,\CPb$. Using this symplectic building block in our construction above, we can obtain the minimal symplectic $4$-manifolds $X_{p,q}$ with the fundamental group isomorphic to $\mathbb Z_{|p|} \times  \mathbb Z_{|q|}$ when $\gcd(p, 2q)=1$, and $e(X_{p,q}) = 10$ and $\sigma(X_{p,q}) = -6$. By replacing $X_1$ by $\mathbb{T}^2 \times \mathbb{S}^2 \#2\,\CPb$ and using the genus two surface of self-intersection $+2$ obtained using two disjoint tori $\mathbb{T}^2 \times pt$ and sphere $S = pt \times \mathbb{S}^2$ in $\mathbb{T}^2 \times \mathbb{S}^2$, we can similarly produce the minimal symplectic $4$-manifolds $Y_{p,q}$ with the fundamental group isomorphic to $\mathbb Z_{|p|} \times  \mathbb Z_{|q|}$ when $\gcd(p, 2q)=1$, and $e(Y_{p,q}) = 9$ and $\sigma(Y_{p,q}) = -5$. In this case, our second building block is $X_2$, and the symplectic genus two surface with self-intersection $-2$ obtained using the tori $T_{1} = \mathbb{T}^2 \times pt$, $T_4 = 4(\mathbb{T}^2 \times pt)$ and sphere $S = pt \times \mathbb{S}^2$ in $\mathbb{T}^2 \times \mathbb{S}^2$. We symplectically blow up three intersection points $p_3$, $p_4$ of $T_3$ and $S$, and symplectically resolve the remaining two intersection points $p_1$ and $p_2$ to obtain a genus two symplectic surface of self-intersection zero in $\mathbb{T}^2 \times \mathbb{S}^2 \#2\,\CPb$.
\end{remark}

\subsubsection{Small symplectic $4$-manifolds with $2\le c_1^2\le 5$}
\

Now we outline the construction of symplectic $4$-manifolds with $c_1^2=2, 3$ and $b^+=1, 2$. To construct such examples, we start with a square zero genus 2 symplectic surface in $\mathbb{T}^4\#2\CPb$. We equip $\mathbb{T}^4=\mathbb{T}^2 \times \mathbb{T}^2$ with a product symplectic form. Next we symplectically resolve the intersection point of the configuration $(\{x\}\times \mathbb T^2)\cup (\mathbb T^2\times \{y\})$ in $\mathbb{T}^4$ to obtain a genus 2 surface of self-intersection 2 in $\mathbb{T}^4$. Symplectically blow up this surface twice to obtain a genus 2 symplectic surface $\hat{\Sigma}_2$ of self-intersection 0 in $\mathbb{T}^4\#2\CPb$. Since $\hat{\Sigma}_2$ has a dual sphere resulting from a blow-up, any meridian of $\hat{\Sigma}_2$  is nullhomotopic in $\mathbb{T}^4\#2\CPb \setminus \hat{\Sigma}_2$. Let $a$, $b$, $c$ and $d$ denote the standard loops that generate $\pi_1(\{x\}\times \mathbb T^2)$ and $\pi_1(\mathbb T^2\times \{y\})$ in $\pi_1(\mathbb T^2\times \mathbb T^2)$ .

Choose  a pair of non-negative integers $p$ and $q$. Let $Y(1/p,1/q)$ denote the symplectic $4$-manifold gotten by performing the following two Luttinger surgeries on $\mathbb{T}^4\#2\CPb$:
\begin{equation}\label{eq:  Y(1/q,1/r)}
(a' \times c', a', -1/p), \ \ 
(b' \times c'', b', -1/q).
\end{equation} 

We take the symplectic fiber sum of $Y(1/p,1/q)$ with the manifolds $Y_{1} = \mathbb{T}^2 \times \mathbb{S}^2 \#4\,\CPb$ (and with $Y_{2}  = \mathbb{T}^2 \times \mathbb{S}^2 \#3\,\CPb$) along the surfaces $\hat{\Sigma}_2$ and $\Sigma_2$ (and $\Sigma_2'$), determined by a map that sends the circles $a$, $b$, $c$, $d$ to $a_1$, $b_1$, $a_{1}^{-1}$, $b_{1}^{-1}$ in the same order. By identifying the meridians $[d ,b ^{-1}]$ and $[a^{-1},d]$ of tori $a' \times c'$ and 
$b' \times c''$ on dual tori (see appendix in \cite{ABP} for the details on this) and applying Seifert-Van Kampen's theorem, the fundamental group of the resulting manifold $Y'$ can be seen to be generated by $a$, $b$, $c$ and $d$, which all commute with each other and the following relations hold: 

\begin{gather*}
[d ,b ^{-1}]=(dad^{-1})^{p},\ \ [a^{-1},d]={b}^{q},\\
ac = 1,\ \ bd = 1,\\
\end{gather*}

Thus $\pi_1(Y')$ is isomorphic to $\mathbb{Z}_p \times \mathbb{Z}_q$. We have $e(Y') = 10$ and $\sigma(X') = -6$. Hence $c_{1}^2 = 2$ and $b^+=1$. If we only perform one Luttinger surgery, then we acquire symplectic manifolds with $b^+=2$, $\pi_1=\mathbb Z\times \mathbb Z_p$ and $c_{1}^2 = 2$, $\chi_{h} = 1$. 

Replacing, $Y_{1}$ with $Y_{2}$ in our construction above, yields symplectic manifolds with $c_{1}^2 = 3$, $\chi_{h} = 1$ and the fundamental groups $\mathbb{Z}_p \times \mathbb{Z}_q$ and $\pi_1=\mathbb Z\times \mathbb Z_p$. 

\begin{remark}Using the building blocks in \cite{AP1, AP}, one can similarly construct symplectic $4$-manifolds with fundamental groups $\mathbb{Z}_p \times \mathbb{Z}_q$ and $2 \leq c_1^2 \leq 5$. The idea is to change the coefficients of the Luttinger surgeries to $-1/p$ and $-1/q$ as in \eqref{eq:  Y(1/q,1/r)}. We leave the details to the interested reader as an exercise (also see the articles \cite{AP, T2} for such examples).  
\end{remark}

To summarize, we have

\begin{theorem}\label{2-7}
For any integer $2\le h\le 3$, there exist minimal symplectic $4$-manifolds $X_{p, q}$ and $X_p$ with $p, q\ge 1$ such that
\begin{enumerate}
\item $\pi_1(X_{p,q})=\mathbb Z_p\times \mathbb Z_q$, $b^+(X_{p,q})=1$ and $c_1^2(X_{p,q})=h$; 
\item $\pi_1(X_{p})=\mathbb Z\times \mathbb Z_p$, $b^+(X_p)=2$ and $c_1^2(X_p)=h$.
\end{enumerate}
\end{theorem}

\section{Cyclic fundamental groups via rational blowdown}

In this section we present a new technique for constructing symplectic $4$-manifolds with cyclic fundamental groups. The following lemma will be a key technical ingredient in our construction, which we call ``the summing spheres''. 

\begin{prop}\label{sewing} Let $X$ be a symplectic $4$-manifold, $T$ is a symplectic torus in $X$ with self-intersection zero, and $\pi_1(X) = 1$. Assume that the spheres $S_1$, $S_2$, ... , $S_p$ ($p>0$) are symplectic submanifolds of $X$ with self-intersections $n_1$, $n_2$, $\cdots$, $n_p$ (where $n_{i} \leq 0$), each $S_i$ has a single transversal intersection with $T$ at distinct points, and all intersections are positive. Then, given the above data, there exists a symplectic $4$-manifold $Y$ with $c_1^2 (Y) = c_1^2 (X)$, $\chi_{h} (Y) = \chi_{h} (X)$, $\pi_1 (Y) = \mathbb{Z}_{p}$ such that $Y$ contains a symplectic sphere of self-intersection $n_{1} + \cdots + n_{p}$.   
\end{prop}

\begin{proof} The manifold $Y$ will be constructed from $X$ by applying Gompf's symplectic sum operation along $T$. Consider $\mathbb{T}^2 \times \mathbb{S}^2$ equipped with a product symplectic form, and let $T_{p}$ denote connected braided symplectic torus representing the homology class $p[\mathbb{T}^2 \times \{ pt \}]$ in $\mathbb{T}^2 \times \mathbb{S}^2$. We form the symplectic sum of $X$ and $\mathbb{T}^2 \times \mathbb{S}^2$ along the tori $T$ and $T_p$, using a gluing map that sends the nullhomotopic circles $a_1$ and $a_2$ of $T$ to the circles $b_1 = b^p$, $b_2 = c$ of $T_p$, in the same order. By Seifert-Van Kampen theorem, the fundamental group of the resulting symplectic manifold $Y$ can be seen to be generated by $b$ and $c$. The following relations hold in $\pi_1(Y)$: $c = 1$ and $b^p = 1$. Thus, $\pi_1(Y)$ is isomorphic to $\mathbb{Z}_p$. The Euler characteristic can be computed as $e(Y) = e(X) + e(\mathbb{T}^2 \times \mathbb{S}^2 ) = e(X)$, and the Novikov additivity gives $\sigma(Y) = \sigma(X) + \sigma(\mathbb{T}^2 \times \mathbb{S}^2 ) = \sigma(X)$. When performing the symplectic sum, we have a freedom to identify $p$ meridional circles $\alpha_i$ of torus $T$ on $S_i$ with $p$ disjoint meridional circles of $T_p$ on the sphere $pt \times \mathbb{S}^2$ (See Figure~\ref{fig:sphere}). Thus, we obtain a closed sphere $S$ with self-intersection $n_{1} + \cdots + n_{p}$. Finally, using the Corollary 1.2 (\cite{Gom}, page 17), we conclude that $S$ is a symplectic submanifold of $Y$.                                                                                                  
\end{proof}

\begin{remark} If we do symplectic sum the other way, e.g. if we sum up one $p$ multiple of the fiber in an elliptic fibration with $\mathbb T^2\times \mathbb S^2$ along fiber, we get a manifold diffeomorphic to the original one. This operation is called smoothly trivial \cite{U1}.
\end{remark}

\begin{figure}
\begin{center}
\includegraphics[scale=.3]{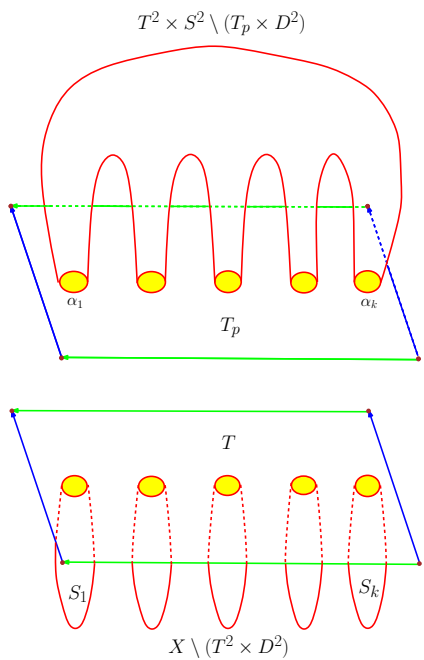}
\caption{Summing Spheres}
\label{fig:sphere}
\end{center}
\end{figure}

\subsection{Elliptic fibrations as words in $\mathbb{SL}(2,\mathbb{Z})$}

Let $f\colon M^4 \to \mathbb{S}^2$ denote a genus $1$ Lefschetz fibration. $M^4$ is $\mathbb{T}^2$ bundle over $\mathbb{S}^2$ except at finitely many critical values of $f$. The fibers over the critical values are called singular fibers of $f$. Furthemore, Lefschetz fibrations are characterized by their global monodromy. This means a factorization of the identity element in the mapping class group $\Gamma_{1}$ of the fiber $\mathbb{T}^2$ as a product of right-handed Dehn twists. The mapping class group $\Gamma_{1} = \mathbb{SL}(2,\mathbb{Z})$ is generated by $t_{a}$, $t_{b}$, and subject to the relations \[ t_{a}t_{b}t_{a}=t_{b}t_{a}t_{b}\qquad {\mbox{ and }} \qquad (t_{a}t_{b})^6=1 \]

\noindent where $t_{a}$ and $t_{b}$ are Dehn twists along the standard generators $a$ and $b$ of the first homology of a torus. 

The classification of the elliptic Lefschetz fibrations is due to Moishezon. He showed that after a possible perturbation an elliptic fibration over $\mathbb{S}^2$ is equivalent to one given by the monodromy $(t_{a}t_{b})^{6n}=1$ in $\Gamma_{1}$. The total space of such fibration is $E(n)$, a simply connected elliptic surface with holomorphic Euler characteristic $\chi_{h} = n$. Furthemore, any elliptic Lefschetz fibration on $E(n)$ admits a sphere section with self--intersection $-n$.

\subsection{Singular fibers and sections in elliptic fibrations}

We briefly discuss the certain singular fibers that occur in elliptic fibrations. We will also discuss the sections of such fibration. 

\medskip

$\bf {Type \ I_{1}}$ : Such a singular fiber is an immersed 2-sphere with one positive double point and referred to as the fishtail fiber. The monodromy of the fishtail fiber is a conjugate of $t_{a}$. 

\medskip

$\bf{Type \ I_{k}}$ : Such a singular fiber is a plumbing of $k$ smooth $2$ spheres of self-intersection $-2$ along a circle and referred to as $I_k$ singularity (or necklace fiber). The monodromy of $I_k$ singularity is a conjugate of $t_{a}^{k}$ ($k\geq 2$). 

\medskip

Persson \cite{Pers} provides a complete classification of the configurations of singular fibers on rational elliptic surfaces. We also refer to \cite{HKK} and \cite{KM} for complete treatment on the topology of elliptic surfaces and their singular fibers. In what follows, we will use an elliptic Lefschetz fibration $E(1) \rightarrow \mathbb{S}^2$ with a singular fiber $F$ of type $I_{5}$, seven singular fibers of type $I_1$. The existence of such fibration is given on page 9 of \cite{Pers}. Here we also provide a proof using a mapping class group argument.

\begin{lemma} 

There exists an elliptic Lefschetz fibration $E(1) \rightarrow \mathbb{S}^2$ with 
a singular fiber $F$ of type $I_{5}$ and seven fishtails.

\end{lemma}

\begin{proof} Using the braid relation several times, it is easy to convert the word $(t_{a}t_{b})^{6}=1$ into the following word

\[ t_at_bt_at_bt_{a}^{3}t_{b}t_{a}^{2}t_{b}^{2} = 1 \]
 
Next, by conjugating, we obtain 

\[ t_at_b(t_at_bt_a^{-1})t_a^5(t_a^{-1}t_{b}t_{a})t_b(t_b^{-1}t_{a}t_{b})t_{b} = 1 \]
 
\noindent The last word translates into the singular fiber of type $I_{5}$ and seven fishtail fibers.
\end{proof}

Moreover, in addition to the information of singular fibers, we would like to know how many disjoint sections such fibration admits. One can construct the fibration with sections explicitly by choosing the cubic polynomials $p_0$ and $p_1$ carefully in the standard picture of obtaining elliptic $E(1)$. Next, we provide two constructions with different elliptic fibration structures.

1. We choose $p_0$ as a multiplication of three degree one polynomials, whose zero set $C_0$ corresponds to three lines $l_{1}$, $l_{2}$, and $l_{3}$ of general position in $\mathbb{CP}^2$. Let the three intersection points be $P, Q, R$. Choose $p_1$ such that (i) the zero set $C_{1}$ corresponding to a smooth elliptic curve; (ii) It shares the common zeros with $p_0$ at $P, Q$ and at other five points other than $R$. The family $t_0p_0+t_1p_1$ for $(t_0, t_1)\in \mathbb{CP}^1$ gives the pencil structure on $\CP$. We could choose a generic $p_1$ satisfying (i) and (ii) such that this family tangent to $C_1$ at points $P, Q$ when $t_0\cdot t_1\neq 0$, and there are at most fishtail singularities in the other fibers. See Figure \ref{fig:config}. Then we blow up consecutively twice at $P$ and $Q$, once at the other five points. We finally achieve an elliptic $E(1)$ with $I_5$ fiber and seven $I_1$. The homology classes for the components of $I_5$ are $H-E_1-E_6-E_7$, $H-E_2-E_8-E_9$, $E_2-E_5$, $H-E_1-E_2-E_3$ and $E_1-E_4$. We thus have seven disjoint sections $E_3, \cdots, E_9$.

\begin{figure}
\begin{center}
\includegraphics[scale=.43]{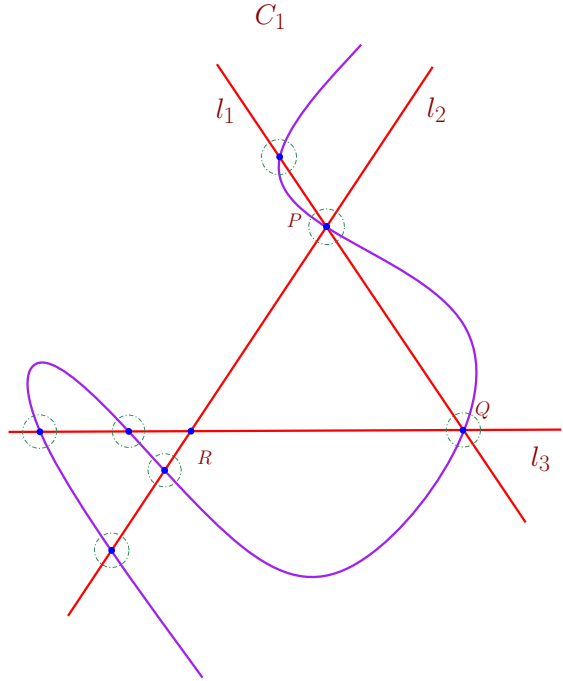}
\caption{Configuration of curves in $\mathbb{CP}^2$}
\label{fig:config}
\end{center}
\end{figure}

2. We choose a cubic $p_0$ such that it's zero set $C_0$ has a unique double point $P$. We also choose another cubic $p_1$ such that it's zero set $C_1$ is smooth and tangent to $C_0$ at the point $P$ with multiplicity five (i.e., this means a $4$-fold tangency to one branch and transverse to the other branch). The other four intersections are in a general position. Notice that such a configuration exists since there is a cubic curve passing through any $6$ given points. Then we choose the $6$ points such that four of them are on one branch near $P$, one point is on the other branch, and the remaining one point is outside of $C_0$. If needed, we perturb the obtained curve such that the other $4$ intersections with $C_0$ are transversal. Now we can find a family of cubics by letting the first $5$ of the $6$ points vary in family and approach to the point $P$. The limiting curve is $C_1$.

 The family $t_0p_0+t_1p_1$ for $(t_0, t_1)\in \mathbb{CP}^1$ gives the pencil structure. We could choose $p_1$ such that this family tangent to $C_1$ at point $P$ (with order five) when $t_0\neq 0$, and there are at most fishtail singularities in the other fibers. See Figure \ref{fig:cubics}.  We do five blow-ups consecutively at $P$ and once at each of the other four. Then we have a $I_5$ fiber with homology classes: $E_1-E_2$, $E_2-E_3$, $E_3-E_4$, $E_4-E_5$ and $3H-2E_1-E_2-E_3-E_4-E_6-E_7-E_8-E_9$. The $-1$ curves $E_5, \cdots, E_9$ are five disjoint sections. Notice that $E_6, \cdots, E_9$ intersect the same component $3H-2E_1-E_2-E_3-E_4-E_6-E_7-E_8-E_9$ and $E_5$ intersects an adjacent component $E_4-E_5$. 

Apparently, the similar construction (by choosing cubic $p_1$ such that the zero set $C_1$ is smooth and tangent to $C_0$ at $P$ with multiplicity four) works to produce $I_4$ fiber with six sections. Five of them intersect the fourth component of $I_4$ and the other intersects the first component. 

\begin{figure}
\begin{center}
\includegraphics[scale=.33]{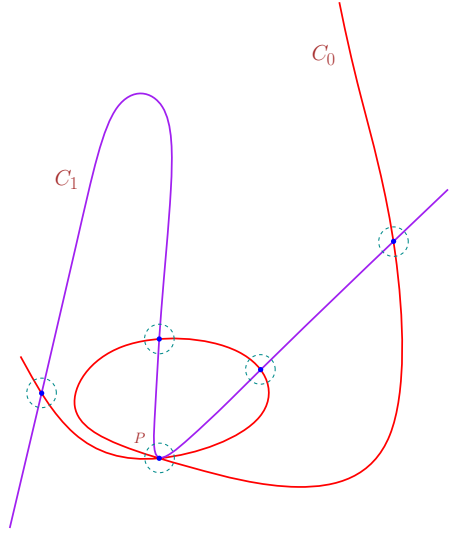}
\caption{Configuration of curves in $\mathbb{CP}^2$}
\label{fig:cubics}
\end{center}
\end{figure}

\subsection{Constructions of symplectic manifolds with $b^+=1$ and $c_1^2>0$}

The final ingredient for our construction is the rational blowdown. Below we recall the definition, and refer the reader to \cite{FS1} for full details

\subsubsection{Rational blowdown}

Let $C_p$ be the smooth $4$-manifold obtained by plumbing $p-1$ disk bundles over the $2$-sphere according to the diagram.

\begin{figure}[h]
\setlength{\unitlength}{0.09in} 
\centering 
\begin{picture}(30,8) 
\put(1, 4){\line(1, 0){11}}
\put(19, 4){\line(1, 0){11}}

\put(0.8, 3.65){\small$\bullet$}
\put(10.8, 3.65){\small$\bullet$}
\put(19.8, 3.65){\small$\bullet$}
\put(29.8, 3.65){\small$\bullet$}

\put(0, 5){$-p-2$}
\put(10, 5){$-2$}
\put(20, 5){$-2$}
\put(29, 5){$-2$}
\put(0, 2.5){$u_0$}
\put(10, 2.5){$u_1$}
\put(20, 2.5){$u_{p-3}$}
\put(29, 2.5){$u_{p-2}$}

\put(13.4, 3.65){$\cdots$}
\put(16, 3.65){$\cdots$}
\end{picture}
\end{figure}

\noindent Here the classes of the spheres have $u_0^2=-(p+2)$ and $u_i^2=-2$, with $i=1, \cdots, p-2$. According to Casson and Harer \cite{CH}, the boundary of $C_p$ is the lens space $L(p^2, p-1)$ which also bounds a rational ball $B_p$ with $\pi_1(B_p) = Z_p$ and $\pi_1(\partial B_p) \rightarrow  \pi_1(B_p)$ surjective. If $C_p$ is embedded in a $4$-manifold $X$ then the rational blowdown manifold $X_p$ is obtained by replacing $C_p$ with $B_p$, i.e., $X_p = (X \setminus C_p) \cup B_p$.  

\begin{lemma}\label{thm:rb} $b_{2}^{+}(X_p) = b_2^{+}(X)$, $\sigma(X_p) = \sigma(X) + (p-1)$, $c_1^{2}(X_p) = c_1^2(X) + (p-1)$, and $\chi_{h}(X_p) = \chi_{h}(X)$.
\end{lemma}

\begin{proof} 

Since the manifold $C_{p}$ is negative definite, we have $b_{2}^{+}(X_{p}) = b_{2}^{+}(X)$ and $b_{2}^{-}(X_{p}) = b_{2}^{-}(X) - (p-1)$. Thus,  $\sigma(X_p) = \sigma(X) + (p-1)$. Using the formulas $c_1^{2} = 3\sigma +2e$ and $\chi_{h} = (\sigma +e)/4$, we have $c_{1}^{2}(X_{p}) = 3\sigma(X_{p}) + 2e(X_{p}) = 3(\sigma(X)+(p-1)) + 2(e(X)-(p-1)) = c_{1}^{2}(X) + (p-1)$ and $\chi_{h}(X_p) = (\sigma(X)+(p-1) + e(X)-(p-1))/4 = \chi_{h}(X)$.  

\end{proof}

It is known that if $X$ admits a symplectic structure for which the spheres are symplectic and intersect positively, then the rational blowdown $X_{p}$ would also admit a symplectic structure induced from the symplectic structures of $X$ and $B_p$ \cite{Sym}. Hence, all the smooth $4$-manifolds we construct below admit symplectic structures.

\subsubsection{An example with $\pi_1=\mathbb Z_5$ and $c_1^2=2$}
\

According to both constructions above $E(1)$ has elliptic fibration with $I_{5}$ singularity, seven fishtail fibers and five disjoint $-1$ sphere sections. Perform a fiber sum with $\mathbb{T}^2 \times \mathbb{S}^2$ using the braided torus $5(\mathbb{T}^2 \times pt)$. By Proposition \ref{sewing}, these five $-1$ sections get glued together, yielding a symplectic $-5$ sphere. Using this $-5$ sphere along with one $-2$ sphere of $I_5$ which intersects only one section, we get the configuration $C_{3}$. By rationally blowing down $C_3$ configuration, we obtain symplectic manifold with $\pi_1=\mathbb Z_5$ and $c_1^2=2$. For the sake of completeness, we spell out the fundamental group computation: using one of the two spheres of $I_5$ fiber adjacent to the above chosen $-2$ sphere, we deform the meridian of $C_3$. Hence, the rational blowdown along $C_3$ kills the fundamental groups $\pi_1(\partial B_3) = {\mathbb{Z}}_{25}$ and $\pi_{1}(B_3) = \mathbb{Z}_{5}$. Thus, the rational blowdown surgery leaves the fundamental group unchanged.

\subsubsection{An example with $\pi_1=\mathbb Z_4$ and $c_1^2=1$}
\

We consider the same elliptic fibration, but use only four $-1$ sections when fiber summing with  $\mathbb{T}^2 \times \mathbb{S}^2$ along the braided torus $4(\mathbb{T}^2 \times pt)$. We choose these $-1$ sections such that one of $-2$ spheres of $I_5$ singularity intersects only one of these sections. By Proposition \ref{sewing}, these four $-1$ sections get glued together to give a $-4$ sphere in a symplectic $4$-manifold with the fundamental group isomorphic to $\mathbb Z_4$. By rationally blowing down $-4$ sphere, we obtian a symplectic $4$-manifold with $c_1^2=1$. Using the fact that the above $-4$ sphere has a dual sphere, with self-intersection $-2$ coming from the necklace fiber $I_5$, we conclude that  $\pi_1=\mathbb Z_4$.

Alternatively, we can obtain symplectic $4$-manifold with $\pi_1=\mathbb Z_4$ and $c_1^2=1$ using the same elliptic fibration, but with four disjoint sections. However, we fiber sum it with  $(\mathbb T^2 \times \mathbb S^2) \#\,\CPb$ along the braided torus $4(\mathbb{T}^2 \times pt)$. The four $-1$ sections are glued with a symplectic $-1$ sphere in class $[pt \times \mathbb S^2]-E$. Summing spheres together, we construct a symplectic $-5$ sphere. Using $-5$ sphere along with $-2$ sphere of $I_5$ intersecting one of the four sections, we have a $C_3$ configuration. By rationally blowing down $C_3$ configuration, we obtain symplectic $4$-manifold with desired properties. 

Another way is using $(\mathbb{T}^2 \times \mathbb{S}^2) \#2\,\CPb$ and a $-2$ sphere in class $[pt\times \mathbb S^2]-E_1-E_2$.  We have a $C_4$ configuration this time if we choose the elliptic fibration in our construction 2. 

One could show that all these three (and other similar constructions) give us the same manifold. The following is a generalization of Gompf's result regarding (rational) blowing down the configuration of a $-4$ sphere with a $-1$ sphere. 

\begin{prop}
Suppose we have following two configurations :
\begin{enumerate}
\item A $C_{n+p}$ configuration along with $p$ spheres of square $-1$ each meeting exactly one points with the $-(n+p+2)$ sphere in $C_{n+p}$. Let $X_1$ be the smooth $4$-manifold obtained by rationally blowing down $C_{n+p}$, and $X_2$ be the smooth $4$-manifold obtained by first blowing down all the $-1$ spheres, then rational blowing down the $C_n$ configuration.

\item A $C_n$ configuration along with one more $-1$ sphere intersecting one point with the last $-2$ sphere in $C_n$. Assume that $X_3$ is obtained by rationally blowing down $C_n$, and $X_4$ is obtained by consecutively blowing down the $-1$ spheres along the reverse order of $C_n$ chain.
\end{enumerate}

Then $X_1\cong X_2$ and $X_3\cong X_4$.
\end{prop}

\subsubsection{An example with $\pi_1=\mathbb Z_4$ and $c_1^2=2$}
\

We start with any elliptic fibration structure on $E(1)$ with at least eight disjoint $-1$ sections. The existence of such fibration follows from the fact that a generic elliptic fibration on $E(1)$ has $9$ disjoint $-1$ sections. We divide these sphere sections into two groups with $4$ spheres in each group. By performing the symplectic summing with $\mathbb{T}^2 \times \mathbb{S}^2$ along the braided torus $4(\mathbb{T}^2 \times pt)$, we glue four $-1$ sections in each group together to get two disjoint $-4$ spheres. By rationally blowing down these two $-4$ spheres, we obtain a symplectic $4$-manifold with $\pi_1=\mathbb Z_4$ and $c_1^2=2$. 

If we only rationally blow down one of the $-4$ sphere, we obtain another example with  $\pi_1=\mathbb Z_4$ and $c_1^2=1$. 

For the sake of completeness, we present the fundamental group computation in details. First, notice that each $-4$ spheres constructed above arises from two disjoint copies of $\mathbb{S}^2 \times pt$ of the trivial fibration on $\mathbb{S}^2 \times \mathbb{T}^2$ via summing the spheres. Hence any torus $T^2$ that descends from $pt \times \mathbb{T}^{2}$ will intersect each of these $-4$ spheres once. In what follows, we will show that there exist a singular $T^2$, consisting of a circular chain of two $-2$ spheres, and each sphere component of this singular $T^2$ intersects only one of these $-4$ spheres. This point will be important in our fundamental group computation. Let us now choose the $T^2$ such that it is disjoint from the braided torus $4(\mathbb{T}^2 \times pt)$. Notice that of the two standard homology generators of the torus, there is one homotopic to a first homology generator of the braided torus. Let us denote this generator by $a$. We are free to choose two circles $a_1$, $a_2$ on $T^2$ that represent the same homology class as $a$ and they separate $2$ intersection points. 
Recall that the braided torus is glued with $E(1)$ along a regular fiber of $E(1)$. We choose vanishing cycles for the circles on the fiber in $E(1)$ corresponding to $a_i$, $1\le i\le 2$. We will still call these curves $a_i$. Since we have $6$ singular fibers in the generic elliptic fibration of $E(1)$ corresponding to the class $a$, we can pick $2$ different singular fibers out of these $6$ singular fibers. We have $2$ disjoint vanishing cycles  intersecting the chosen fiber at $a_i$ respectively. Using these vanishing cycles and their corresponding vanishing discs, we construct two spheres $S(a_1, a_2)$ and $S(a_2, a_1)$ in the symplectic $4$-manifold $X$. Each of these spheres contains one and only one intersection points with $-4$ spheres as mentioned above. Hence the meridians of these two $-4$ spheres would be trivial in $\pi_1(X)$. Since $\pi_1(X)=\mathbb Z_4$ by Proposition \ref{sewing}, the fundamental group of the manifold after rationally blowing down one or two $-4$ spheres is still $\mathbb Z_4$.

\subsubsection{An example with $\pi_1=\mathbb Z_6$ and $c_1^2=3$}

Consider an elliptic fibration structure on $E(1)$ with one $I_4$ singularity, eight fishtail fibers and six disjoint $-1$ sphere sections. Notice the six sections can be chosen such that five of them intersect the fourth component of $I_4$, and the remaining sphere intersects the first component (as in construction 2 of Section 4.2). 

Next we perform a fiber sum with $\mathbb{T}^2 \times \mathbb{S}^2$ using the braided torus $6(\mathbb{T}^2 \times pt)$ to sew the six $-1$ sections together get a $-6$ sphere. Using the first and second $-2$ spheres of $I_4$, we obtain the configuration $C_{4}$ for our rational blowdown. By rationally blowing down the $C_4$ configuration, we obtain a symplectic $4$-manifold with $c_1^2=3$ and $\pi_1=\mathbb Z_6$.

\subsection{Constructions of symplectic $4$-manifolds with $b^+=3$ and $c_1^2>0$}

Using the same techniques, we can also construct symplectic $4$-manifolds with higher $b^+$ and cyclic fundamental groups. To illustrate this, we will mention three examples with $b^+=3$ and $c_1^2>0$. 

\subsubsection{An example with $\pi_1=\mathbb Z_2$ and $c_1^2=3$}

Let us start with an elliptic fibration structure on $E(1)$ with one $I_6$ singularity, two $I_2$ fibers, two fishtails, and six disjoint $-1$ sphere sections. These $-1$ sections can be chosen such that each of them intersect a different component of $I_6$ fiber. By performing the fiber of two copies of $E(1)$ along a regular fiber, we obtain an elliptic fibration structure on $E(2)$ with two $I_6$ fibers, four $I_2$ fibers, four fishtails, and six disjoint $-2$ spheres. Moreover, it is easy to see that each $-2$ sphere sections has two disjoint dual $-2$ spheres, arrising from the components of two $I_6$ fibers in $E(2)$. Let us group these six $-2$ sphere sections into three pairs. By performing fiber sum with $\mathbb{T}^2 \times \mathbb{S}^2$ along the braided torus $2(\mathbb{T}^2 \times pt)$, we obtain a minimal symplectic $4$-manifold with $b^+=3$, $c_1^2=0$ and $\pi_1=\mathbb Z_2$, which contains three disjoint $-4$ spheres. Moreover, it follows from the above discussion that each  $-4$ sphere has a dual $-2$ sphere. By rationally blowing these three $-4$ spheres, we obtain a minimal symplectic $4$-manifold with $b^+=3$, $c_1^2=3$ and $\pi_1=\mathbb Z_2$. 

\subsubsection{An example with $\pi_1=\mathbb Z_3$ and $c_1^2=3$}

We will again make use of the above elliptic fibration structure on $E(2)$ with two $I_6$ fibers, four $I_2$ fibers, four fishtails, and six disjoint $-2$ sphere sections. Let us group three of the six $-2$ sphere sections together. By performing fiber sum of $E(2)$ with $\mathbb{T}^2 \times \mathbb{S}^2$ along the braided torus $3(\mathbb{T}^2 \times pt)$, we obtain a minimal symplectic $4$-manifold with $b^+=3$, $c_1^2=0$ and $\pi_1=\mathbb Z_3$, which contains symplectic $-6$ sphere. Furthemore, using this $-6$ sphere and two consecutive $-2$ spheres of $I_6$ fiber, we obtain a configuration of $C_{4}$ for a rational blowdown.  Notice that the last $-2$ sphere of $C_4$ has a dual $-2$ sphere, which comes from the remaining spheres of $I_6$ fiber. By rationally blowing down $C_4$ configuration, we obtain a minimal symplectic $4$-manifold with $c_1^2=3$ and $\pi_1=\mathbb Z_3$.

 \subsubsection{An example with $\pi_1=\mathbb Z_2$ and $c_1^2=4$}
We start with an elliptic fibration structure on $E(2)$ with twenty four fishtail fibers and at least eight disjoint $-2$ sphere sections. Let us assume that half of the fishtail fibers have the vanishing cycles $a$ and $b$ in the notation given above.  We group the sections into four pairs. Performing fiber sum with $\mathbb{T}^2 \times \mathbb{S}^2$ along the braided torus $2(\mathbb{T}^2 \times pt)$ yields minimal symplectic 4-manifold $X$ with $\pi_1=\mathbb Z_2$ and $c_1^2=0$. Moreover, this symplectic $4$-manifold contains four disjoint symplectic $-4$ spheres.  After rationally blowing down all these spheres, we obtain a symplectic $4$-manifold with $b^+=3$, $c_1^2=4$ and $\pi_1=\mathbb Z_2$. For the sake of completeness, we explain the computation of the fundamental group in details. Notice that each $-4$ spheres arises from four different $S^2$ of the trivial fibration on $\mathbb{S}^2 \times \mathbb{T}^2$. If we pick a $T^2$ in that fibration, it will intersect each of these $S^2$ (and consequently $-4$ spheres mentioned above) once. Furthermore, this $T^2$ can be chosen in such a way that it is disjoint from the braided torus $2(\mathbb{T}^2 \times pt)$. Recall that of the two standard first homology generators of the torus, there is one homotopic to a first homology generator of the braided torus. We will denote this generator by $a$. We are free to choose four circles $a_1$, $a_2$, $a_3$, $a_4$ on $T^2$ each representing the homology class $a$ such that they separate $4$ intersection points. Notice that the braided torus in $\mathbb{T}^2 \times \mathbb{S}^2$ is glued with a regular fiber of an elliptic fibration on $E(2)$, and thus we choose vanishing cycles for the circles on the fiber in $E(2)$ corresponding to $a_i$, $1\le i\le 4$. We will still call these curves $a_i$. Since we have $12$ singular fibers in the generic elliptic fibration of $E(2)$ corresponding to the class $a$. Take $4$ different singular fibers out of these $12$, which will give us $4$ disjoint vanishing cycles intersecting the chosen fiber at $a_i$ respectively. Furthemore, using the given odering of these vanishing cycles and their corresponding vanishing discs, we  construct  four matching spheres $S(a_1, a_2)$, $S(a_2, a_3)$, $S(a_3, a_4)$, and $S(a_4, a_1)$ in the symplectic $4$-manifold $X$. Each of these spheres contains one and only one intersection points with $-4$ spheres as mentioned above. Hence the meridians of these four $-4$ spheres would be trivial in $\pi_1(X)$. Since $\pi_1(X)=\mathbb Z_2$, the fundamental group of the manifold after rationally blowing down the four $-4$ spheres is still $\mathbb Z_2$.




\section*{Acknowledgments} A. Akhmedov was partially supported by NSF grants FRG-1065955 and DMS-1005741. W. Zhang was partially supported by AMS-Simons travel grant. We would like to thank I. Dolgachev for helpful discussions. We are grateful to Tian-Jun Li for suggestions and for his kind encouragement. The authors are also thankful to Rafael Torres for his remarks and suggestions. Finally, we would like to thank the anonymous referees who gave us many constructive suggestions and valuable comments on the previous versions of this paper, which led to the improvement of the presentation of the material in the paper.

\end{document}